\documentclass{article}

\usepackage{hyperref}
\usepackage{amssymb,amsthm}
\usepackage[textsize=small]{todonotes}
\usepackage{upref,setspace}
\usepackage{xcolor,colortbl}
\usepackage{enumerate}
\usepackage{bbm}
\usepackage{float}
\usepackage{bm}
\usepackage[mathscr]{eucal}
\usepackage{graphicx}
\usepackage{sidecap}
\usepackage{latexsym}
\usepackage{psfrag}
\usepackage{epsfig}
\usepackage{enumerate}
\usepackage{amsmath}
\usepackage{multirow}
\usepackage{comment}
\usepackage{amsfonts}
\usepackage{color}
\usepackage{nicematrix}
\usepackage[textsize=small]{todonotes}
\usepackage{mathtools}

\usepackage{epstopdf}
\usepackage{subcaption}

\graphicspath{{C:/Users/Stef/Dropbox/CyclicalLongMem/MatlabCLM/Plots/}}

\mathtoolsset{showonlyrefs}

\numberwithin{equation}{section}
\numberwithin{figure}{section}
\numberwithin{table}{section}
\newtheorem{theorem}{Theorem}[section]
\newtheorem{corollary}[theorem]{Corollary}
\newtheorem{lemma}[theorem]{Lemma}

\newtheorem{defn}[theorem]{Definition}

\newtheorem{proposition}[theorem]{Proposition}

\newtheorem{remark}[theorem]{Remark}

\numberwithin{equation}{section}

\newcommand{\veps}{\varepsilon}

\newcommand{\EE}{\mathbb{E}}
\newcommand{\RR}{\mathbb{R}}
\newcommand{\ZZ}{\mathbb{Z}}
\newcommand{\NN}{\mathbb{N}}

\newcommand{\1}{\mathbf{1}}

\newcommand{\cli}{\mathcal{I}}

\newcommand{\om}{\omega}

\definecolor{expcol}{rgb}{1.0,0.5,0.0}
\definecolor{ecol}{rgb}{0.0, 0.0, 1.0}

\usepackage{authblk}

\begin{document}

\title{Cyclical Long Memory:\\  Decoupling, Modulation, and Modeling}
\author[1]{Stefanos Kechagias}
\author[2]{Vladas Pipiras}
\author[2]{Pavlos Zoubouloglou\footnote{Corresponding author, e-mail: pavlos@ad.unc.edu.}}
\affil[1]{Department of Statistics, AUEB}
\affil[2]{Department of Statistics and Operations Research, UNC Chapel Hill}

\maketitle

\begin{abstract}
A new model for general cyclical long memory is introduced, by means of random modulation of certain bivariate long memory time series. This construction essentially decouples the two key features of cyclical long memory: quasi-periodicity and long-term persistence. It further allows for a general cyclical phase in cyclical long memory time series. Several choices for suitable bivariate long memory series are discussed, including a parametric fractionally integrated vector ARMA model. The parametric models introduced in this work have explicit autocovariance functions that can be readily used in simulation, estimation, and other tasks.\\


 \end{abstract}

\section{Introduction} \label{sec:intro}

The main goal of this work is to shed further light on the phenomenon of the so-called cyclical long memory concerning stationary time series. Let $X = \{X_n\}_{n \in \ZZ}$ be a second-order stationary time series with zero mean $\EE X_n = 0$ for simplicity and autocovariance function (ACVF) $\gamma_X(h) = \EE X_h X_0,\; h \in \ZZ$. If it exists, denote the spectral density of $X$ by $f_X(\lambda), \lambda \in (-\pi,\pi)$. Under mild assumptions (and by the chosen convention), one expects that $ f_X(\lambda) = (2\pi)^{-1} \sum_{h=-\infty}^\infty e^{-ih\lambda} \gamma_X (\lambda),\; \lambda \in (-\pi,\pi).$ By symmetry, it is enough to focus on $f_X(\lambda), \;\lambda \in (0,\pi)$.

The time series $X$ exhibits \textit{Cyclical Long Memory} (CLM, for short) if
\begin{equation}
\label{eq:cycl-ACVF}
\begin{split}
    \gamma_X(h) &\simeq c_\gamma \cos(\lambda_0 h + \phi) h^{2d-1} \\
    &= c_{\gamma,1} \cos(\lambda_0 h) h^{2d-1} + c_{\gamma,2} \sin(\lambda_0 h) h^{2d-1}, \quad \text{as} \; h \to \infty,
\end{split}
\end{equation}
where $d \in (0,1/2)$, $\lambda_0 \in (0,\pi), c_\gamma>0,c_{\gamma,1} > 0$, $c_{\gamma,2} \in \RR$, and $\phi \in (-\frac{\pi}{2},\frac{\pi}{2})$. The notation $\gamma_1(h) \simeq \gamma_2(h) $ in \eqref{eq:cycl-ACVF} and throughout this work will stand for $\gamma_1(h) - \gamma_2(h) = o(h^{2d-1})$, as $h\to\infty$. A formal definition is given and investigated in Appendices \ref{app:sec-defn}, \ref{app:sec-spec-to-time} and \ref{app:sec-time-to-spec}. The parameter $d$ is known as \textit{long memory} parameter, and we will refer to $\lambda_0$ and $\phi$ as \textit{cyclical frequency} and \textit{cyclical phase} parameters. In fact, as shown in Appendix \ref{subsubsec:admis-phase}, the cyclical phase 
\begin{equation} \label{eq:admiss-sets-def}
    \phi \in \cli_d \doteq \left[\left(d-\frac{1}{2}\right)\pi, \left(\frac{1}{2}-d\right)\pi\right],
\end{equation}
which we refer to as the \textit{set of admissible cyclical phase} parameters. The boundary points $\phi = \pm \left( \frac{1}{2} - d \right) \pi$ of the set $\cli_d$ and the value $\phi = 0$ play a special role as indicated below.

In the spectral domain, CLM can be thought of as 
\begin{equation} \label{eq:cyclical-spectrum}
 f_X(\lambda) \sim 
 \begin{cases}
       c_f^+ (\lambda - \lambda_0)^{-2d}, &\quad\text{as}\; \lambda \to \lambda_0^+,\\
       c_f^- (\lambda_0 - \lambda)^{-2d}, &\quad\text{as}\; \lambda \to \lambda_0^-,
 \end{cases}
\end{equation}
where $c_f^+,c_f^- \ge 0$ are constants so that $c_f^+ + c_f^- > 0$. Namely, the spectral density diverges around the cyclical frequency $\lambda_0$. As usual, $\sim$ denotes asymptotic equivalence. Conditions and results for going from \eqref{eq:cyclical-spectrum} to \eqref{eq:cycl-ACVF}, and vice versa, are given in Appendices \ref{app:sec-spec-to-time} and \ref{app:sec-time-to-spec} below, along with relations connecting $c_f^+,c_f^-$ and $c_\gamma,\phi$. These results are of independent interest and, to the best of our knowledge, are new in this area. In particular, we have that
\begin{equation} \label{eq:cycl-phi-0}
\phi = 0 \quad \text{if and only if} \quad c_f^+ = c_f^- \doteq c_f,
\end{equation}
so that \eqref{eq:cyclical-spectrum} becomes $f_Y(\lambda) \sim c_f |\lambda - \lambda_0|^{-2d}, \;\text{as}\; \lambda \to \lambda_0,$ that is, the divergence around $\lambda = \lambda_0$ is symmetric. 
In particular, the boundary points satisfy the relation:
\begin{equation} \label{eq:cycl-phi-boundary}
\phi = \pm \left( \frac{1}{2} - d  \right) \pi \quad \text{if and only if} \quad c_f^{\mp} = 0.
\end{equation}
But we also caution the reader that the ``boundary" case \eqref{eq:cycl-phi-boundary} needs to be treated more carefully, as explained in Appendices \ref{app:subsec-spec-to-time-boundary} and \ref{app:subsec-time-to-spec-boundary}. We note in that regard that, by our convention, for example, $c_f^+ = 0$ in \eqref{eq:cyclical-spectrum} stands for $(\lambda - \lambda_0)^{2d} f_X(\lambda) \to 0$ as $\lambda \to \lambda_0^+$. It subsumes more refined asymptotics as $f_X(\lambda) \sim c(\lambda - \lambda_0)^{-2\delta}$ with $\delta < d$, which need to be taken into account as presented in, e.g., Appendix \ref{app:subsec-spec-to-time-boundary}. The models considered below will satisfy both \eqref{eq:cycl-ACVF} and \eqref{eq:cyclical-spectrum}. When $\lambda_0 = 0$, CLM becomes the usual well-studied Long Memory (LM) (e.g., \cite{BerFenGhoSuc13book,GirKouSur12Book,SamBook16,pipiras_taqqu_2017}), but the focus here is on $\lambda_0 > 0$. CLM is sometimes also called seasonal LM, though some authors use the latter term more specifically for the so-called seasonal FARIMA model.

The origins of CLM go back at least to the celebrated paper by Hosking \cite{Hos81}, who considered fractional differencing in modeling LM, but in the last paragraph of the paper, also discussed briefly a CLM model exhibiting ``both long-term persistence and quasiperiodic behavior." The model became known as the Gegenbauer series and has drawn most of the attention in this area, especially as likely being the most popular parametric model. Various developments and extensions for this model (ARMA counterparts, multiple divergence frequencies, estimation, multivariate constructions, and so on) can be found in, e.g., \cite{Anděl86,GraZhaWoo89,GraZhaWoo94,GirLei95,WooCheGra98,DisPeiPro16,WuPei18,Espleooleand15}. A review can be found in \cite{Dis18}. We note that the corresponding models usually exhibit CLM with cyclical phase $\phi = 0$ only and that their ACVFs are generally challenging to compute \cite{arteche1999,Tuc12}. Other aspects of CLM (limit theorems, semi-parametric estimations, and so on) were considered in \cite{IvaLeo13,Ole13,Oul02,OulPhi11,GiHiRo01,arteche1999,Arteche20002,Arteche20003}. 

Our own interest in CLM stems from its relevance to naval applications \cite{Pipzousap22}. Stationary spatio-temporal models are available in Naval Architecture, Oceanography and related fields for the wave height at given spatial locations and over time. The Longuet-Higgins model is a celebrated example \cite{Lonseldea57} and involves a spectrum function, which can be thought of as the spectrum of the wave height process at a fixed location. The process has short memory. However, if the wave height is measured at a location traveling at a constant speed, the resulting process can be shown to exhibit CLM. This is relevant to ships traveling at constant speed, whose motions then inherit the CLM from that of the associated wave excitation. Aspects of this behavior have been known in Naval Architecture for years (e.g., \cite{DenPie54,Lew89}), well before the paper by Hosking \cite{Hos81}, but connections to CLM have been clarified only recently in \cite{Pipzousap22}. As far as we know, this seems to be the first physical model of CLM in the sense of being constructed from first (physics) principles. We also note that the CLM phenomenon in the naval applications is associated with a non-zero cyclical phase $\phi \neq 0$.

A particularly interesting feature of the definition \eqref{eq:cycl-ACVF} is its multiplicative nature as the product of $\cos(\lambda_0 h + \phi)$ and $h^{2d-1}$. The term $h^{2d-1}$ corresponds to the usual LM behavior, while the term $\cos(\lambda_0 h + \phi)$ is associated with modulation. This suggests that CLM series could be viewed as modulated LM series, where the LM and cyclical effects are decoupled. In fact, when $\phi = 0$, such modulated series exhibiting CLM are straightforward to construct. Take two independent copies $\{Y_{1,n}\}_{n \in \ZZ},\{Y_{2,n}\}_{n \in \ZZ}$ of a zero mean LM series $Y$ satisfying $\gamma_Y(h) \sim c_\gamma h^{2d-1}$, as $h \to \infty$. Set
\begin{equation}\label{eq:cycl-constr-basic}
    X_n \doteq \cos(\lambda_0 n) Y_{1,n} + \sin(\lambda_0n) Y_{2,n}, \quad n \in \ZZ.
\end{equation}
By construction,
\begin{equation} \label{eq:acvf-X-basic}
\begin{split}
    \EE X_{n+h} X_{n} &= \big( \cos(\lambda_0 h) \cos(\lambda_0 (n+h)) + \sin(\lambda_0 h) \sin(\lambda_0 (n+h))   \big) \gamma_Y(h) \\
    &= \cos(\lambda_0 h) \gamma_Y(h) \sim c_\gamma \cos(\lambda_0 h) h^{2d-1}, \quad \text{as}\; h \to \infty,
\end{split}
\end{equation}
satisfies \eqref{eq:cycl-ACVF} with $\phi = 0$. In the CLM literature, this construction was studied and exploited in applications in \cite{ProMad22,MadPro22}, where the model \eqref{eq:cycl-constr-basic} was termed the \textit{fractional sinusoidal waveform process}.

The model \eqref{eq:cycl-constr-basic} has in fact a long history in signal processing, where its construction is known as \textit{random modulation} (RMod, for short). See, for example, a review paper \cite{Pap83}. A textbook on probability for engineers \cite{PapBook} also includes a section (Section 11.3) on this topic. In the context of random modulation, the series $\{Y_{1,n}\}_{n \in \ZZ}$ and $\{Y_{2,n}\}_{n \in \ZZ}$ need not be independent for the RMod model \eqref{eq:cycl-constr-basic} to yield a stationary series. A sufficient condition is for the vector series
\begin{equation} \label{eq:Y_n-def}
Y_n = \begin{pmatrix}
    Y_{1,n} \\ Y_{2,n}
\end{pmatrix}
\end{equation}
to be second-order stationary with the matrix ACVF
\begin{equation} \label{eq:acvf-Y-general}
    \gamma_{Y}(h) = \EE Y_{n+h} Y_n^T = \begin{pmatrix}
        \EE Y_{1,n+h} Y_{1,n} & \EE Y_{1,n+h} Y_{2,n} \\
        \EE Y_{2,n+h} Y_{1,n} & \EE Y_{2,n+h} Y_{2,n}
    \end{pmatrix} \doteq 
    \begin{pmatrix}
        \gamma_{Y,11}(h) & \gamma_{Y,12}(h) \\
        \gamma_{Y,21}(h) & \gamma_{Y,22}(h)
    \end{pmatrix}
\end{equation}
satisfying
\begin{equation} \label{Property1} \tag{P-T}
    \gamma_{Y,11}(h) = \gamma_{Y,22}(h) \; \text{and} \; \gamma_{Y,12}(h) = - \gamma_{Y,21}(h) \; \text{for all}\; h \in \ZZ.
\end{equation} 
Under \eqref{Property1}, the ACVF of the RMod model \eqref{eq:cycl-constr-basic} can be checked to be (see the proof of Proposition \ref{thm:constr} below)
\begin{equation} \label{eq:acvf-X-model}
    \gamma_X(h) = \cos(\lambda_0 h) \gamma_{Y,11}(h) - \sin(\lambda_0 h) \gamma_{Y,12}(h).
\end{equation}
Papoulis \cite{Pap83} reviews a number of questions addressed in the past regarding the RMod series \eqref{eq:cycl-constr-basic} under the assumption \eqref{Property1}. We will draw connections to this literature below, but one important difference is that the focus here will be on LM series. The property \ref{Property1} is reformulated in the spectral domain in Proposition \ref{prop:spectral-property} below.

For uncorrelated LM series $\{Y_{1,n}\}_{n \in \ZZ}$ and $\{Y_{2,n}\}_{n \in \ZZ}$, the expression \eqref{eq:acvf-X-model} reduces to \eqref{eq:acvf-X-basic} and leads to CLM \eqref{eq:cycl-ACVF} with $\phi = 0$. The presence of both cosine and sine in \eqref{eq:acvf-X-model} suggests that \textit{correlated} LM series $\{Y_{1,n}\}_{n \in \ZZ}$ and $\{Y_{2,n}\}_{n \in \ZZ}$ might lead to the RMod series \eqref{eq:cycl-constr-basic} exhibiting CLM with general cyclical phase $\phi$. This leads to the following main questions addressed in this work:
\begin{description}
    \item[Q1:] Are there correlated LM series $\{Y_{1,n}\}_{n \in \ZZ}$ and $\{Y_{2,n}\}_{n \in \ZZ}$ (or bivariate LM series $Y$ in \eqref{eq:Y_n-def}) having property \eqref{Property1} such that the RMod series \eqref{eq:cycl-constr-basic} is CLM satisfying \eqref{eq:cycl-ACVF} with general $\phi$?

    \item[Q2:] Are there parametric bivariate LM series $Y$ in Q1 which are suited for modeling CLM and related tasks, especially for resulting CLM models having explicit ACVF?
\end{description}
We show in this work that the answers to both questions are affirmative. More specifically, we characterize the bivariate LM series in Q1 and establish conditions for the resulting RMod model \eqref{eq:cycl-constr-basic} to have CLM with a particular phase $\phi$. We also propose a parametric model for such series $Y$ with explicit ACVF, and hence the same for the RMod model \eqref{eq:cycl-constr-basic} in view of the relation \eqref{eq:acvf-X-model}. This offers computational and modeling advantages over the Gegenbauer models discussed above. Various additional contributions regarding this construction and CLM will also be made, where we would draw the reader's attention to the extensions of our approach allowing for the exponents in \eqref{eq:cyclical-spectrum} as $\lambda \to \lambda_0^+$ and $\lambda \to \lambda_0^-$ to be different, that is, 
\begin{equation} \label{eq:cyclical-spectrum-assym}
 f_X(\lambda) \sim 
 \begin{cases}
       c_{f,+} (\lambda - \lambda_0)^{-2d_+}, &\quad\text{as}\; \lambda \to \lambda_0^+,\\
       c_{f,-} (\lambda_0 - \lambda)^{-2d_-}, &\quad\text{as}\; \lambda \to \lambda_0^-,
 \end{cases}
\end{equation}
where $c_{f,+}, c_{f,-} > 0$ and $d_-,d_+ \in (0,1/2)$. We naturally construct such series $X$ as $X = X^+ + X^-$ by taking uncorrelated $X^+,X^-$ satisfying \eqref{eq:cyclical-spectrum} with 
\begin{equation} \label{eq:cyclical-spectrum-assym-constr}
 f_{X^+}(\lambda) \sim 
 \begin{cases}
       c_{f,+} (\lambda - \lambda_0)^{-2d_+}, &\quad\text{as}\; \lambda \to \lambda_0^+,\\
       0 \cdot (\lambda_0 - \lambda)^{-2d_+}, &\quad\text{as}\; \lambda \to \lambda_0^-,
 \end{cases}, \quad
  f_{X^-}(\lambda) \sim 
 \begin{cases}
       0  \cdot  (\lambda - \lambda_0)^{-2d_-}, &\quad\text{as}\; \lambda \to \lambda_0^+,\\
       c_{f,-} (\lambda_0 - \lambda)^{-2d_-}, &\quad\text{as}\; \lambda \to \lambda_0^-.
 \end{cases}
\end{equation}
We need, however, to take into account the delicate issues around the ``boundary" cases in \eqref{eq:cyclical-spectrum-assym-constr} as noted following \eqref{eq:cycl-phi-boundary}.

The rest of the paper is organized as follows. In Section \ref{sec:prelim}, we recall some facts about bivariate LM series and provide some results on bivariate series satisfying \eqref{Property1}. In Section \ref{sec:cycl-constru}, we establish connections between bivatiate LM and CLM series through the RMod construction \eqref{eq:cycl-constr-basic} and provide a parametric model for the bivariate and CLM series. Section \ref{sec:extensions} extends the CLM construction to models with different memory parameters as in \eqref{eq:cyclical-spectrum-assym} and to models with multiple singularities. Section \ref{sec:numerical} provides numerical illustrations. Section \ref{sec:conclusions} concludes. Appendices \ref{app:sec-defn}--\ref{app:sec-auxillary} gather results relating the definitions \eqref{eq:cycl-ACVF} and \eqref{eq:cyclical-spectrum} of CLM.

\section{Preliminaries} \label{sec:prelim}

As discussed in Section \ref{sec:intro}, our CLM model will involve bivariate LM series and RMod. We recall basic facts about bivariate LM in Section \ref{subsec:mlm} below. We relate the parameters of the two CLM definitions in Section \ref{subsec:connec-clm-def}. We also provide one general construction of bivariate series that satisfy property \eqref{Property1} and hence that can be used in the RMod model \eqref{eq:cycl-constr-basic}; see Section \ref{subsec:random-modu}. This construction will be adapted to bivariate LM series in Section \ref{sec:cycl-constru} to construct a parametric model of CLM.

\subsection{Bivariate Long Memory} \label{subsec:mlm}

Let $Y_n = (Y_{1,n},Y_{2,n})^T, n \in \ZZ,$ be a bivariate second-order stationary time series with matrix ACVF $\gamma_Y(h), h \in \ZZ$, in \eqref{eq:acvf-Y-general}. Again, we assume $\EE Y_n = 0$ for simplicity. The respective spectral density will be denoted $f_Y(\lambda), \lambda \in (-\pi,\pi)$. Under mild assumptions, one expects $f_Y(\lambda) = (2\pi)^{-1} \sum_{h=-\infty}^\infty e^{-i \lambda h} \gamma_Y(h)$. 

The time series $Y$ is called \textit{bivariate LM} if one of the following two conditions is satisfied:
\begin{itemize}
    \item Time domain: as $h \to \infty$,
       \begin{equation} \label{eq:biv-gamma-general}
           \gamma_Y(h) = \begin{pmatrix}
               \gamma_{Y,11}(h) & \gamma_{Y,12}(h) \\
               \gamma_{Y,21}(h) & \gamma_{Y,22}(h)
           \end{pmatrix}
           \sim \begin{pmatrix}
               R_{11} h^{2d_1-1} & R_{12} h^{d_1 + d_2 -1} \\
               R_{21} h^{d_1 + d_2 -1} & R_{22} h^{2d_2-1}
           \end{pmatrix},
       \end{equation}
      where $d_j \in (0,1/2)$, $R_{jk} \in \RR, j,k=1,2$, and $R_{11},R_{22} > 0$.

    \item Spectral domain: as $\lambda \to 0^+$,
    \begin{equation} \label{eq:biv-spectral-general}
        f_Y(\lambda) = \begin{pmatrix}
            f_{Y,11}(\lambda) & f_{Y,12}(\lambda) \\
             f_{Y,21}(\lambda) & f_{Y,22}(\lambda)
        \end{pmatrix}
        \sim \begin{pmatrix}
               g_{11} \lambda^{-2d_1} & (g_{12}e^{-i\om}) \lambda^{-d_1 -d_2 } \\
               (g_{12}e^{i\om}) \lambda^{-d_1 - d_2 } & g_{22} \lambda^{-2d_2}
           \end{pmatrix},
    \end{equation}
    where $d_{j} \in (0,1/2),j=1,2$, $g_{11},g_{22} > 0, g_{12} \in \RR,$ and $\om \in (-\pi,\pi)$.
\end{itemize}

Under mild assumptions, the time-domain and spectral-domain definitions above can be shown to be equivalent (Proposition 2.1 in \cite{KecPi15}). The assumptions will be satisfied for the models considered below, so we will use them interchangeably. There are also explicit formulas relating $R_{11},R_{12},R_{21},R_{22}$ and $g_{11},g_{12},g_{21},\om$ in \cite{KecPi15}.

The parameter $\om$ in \eqref{eq:biv-spectral-general}  is called the \textit{bivariate phase parameter}. It controls the (a)symmetry of the series at large lags, that is, one has $\om = 0$ if and only if $R_{12} = R_{21}$ in \eqref{eq:biv-gamma-general}. The parameter $\om$ should not be confused with the cyclical phase parameter $\phi$ in \eqref{eq:cycl-ACVF}, even if both are related to symmetry (see \eqref{eq:cycl-phi-0} for $\phi$).

A way to construct bivariate LM series is through two-sided linear representations of the form
\begin{equation} \label{eq:Zn}
    Y_n = \sum_{l= - \infty}^\infty A_l \veps_{n-l}, \quad n \in \ZZ, 
\end{equation}
where $\{\veps_n\}_{n \in \ZZ}$ is a white noise series (i.e., $\EE \veps_n = 0$, $\EE \veps_n \veps_m^T = 0, n \neq m$, and $\EE \veps_n \veps_n^T = I_2$) and $A_l \in \RR^{2 \times 2}$ are such that 
\begin{equation} \label{eq:Al-asympt}
    A_l \sim \begin{cases}
        \begin{pmatrix}
            A_{11}^+ l^{d_1-1} & A_{12}^+ l^{d_1-1} \\
            A_{21}^+ l^{d_2-1} & A_{22}^+ l^{d_2-1}
        \end{pmatrix}, \quad l \to + \infty, \vspace{2mm} \\
        \begin{pmatrix}
            A_{11}^- (-l)^{d_1-1} & A_{12}^- (-l)^{d_1-1} \\
            A_{21}^- (-l)^{d_2-1} & A_{22}^- (-l)^{d_2-1}
        \end{pmatrix}, \quad l \to - \infty.
    \end{cases}
\end{equation}
There are similar formulas relating $A^{\pm}_{jk}$ in \eqref{eq:Al-asympt} to $R_{j,k}$ in \eqref{eq:biv-gamma-general} (Proposition 3.1 in \cite{KecPi15}).

\subsection{Connections Between CLM Definitions} \label{subsec:connec-clm-def}

In this section, we relate the parameters $c_\gamma,\phi$ in \eqref{eq:cycl-ACVF} with $c_f^+,c_f^-$ in \eqref{eq:cyclical-spectrum}. We prove these relations in the appendix, with special treatment for the ``boundary" case. One can obtain $c_\gamma,\phi$ from $c_f^+,c_f^-$ by
\begin{equation} \label{eq:cgamma-psi}
   \begin{split}
    c_\gamma &\doteq 2\Gamma(1-2d)\sqrt{(c_f^+)^2 + (c_f^-)^2 - 2 c_f^+ c_f^-\left(\cos( 2 \pi d) \right)^2}, \\
    \phi &\doteq  \arcsin \left( \frac{(c_f^+ - c_f^-) \cos(\pi d)}{\sqrt{(c_f^+)^2 + (c_f^-)^2 - 2 c_f^+ c_f^-\left(\cos(2 \pi d) \right)^2}}  \right).
   \end{split}
\end{equation}
See Appendix \ref{app:sec-spec-to-time}. Moreover, one can obtain $c_f^+,c_f^-$ from $c_\gamma,\phi$ by
\begin{equation} \label{eq:cfpm-cgamma}
    c_f^\pm \doteq \frac{c_\gamma}{2\pi} \Gamma(2d) \cos(\pi d \mp \phi).
\end{equation}
See Appendix \ref{app:sec-time-to-spec}.

\subsection{Bivariate Series for Random Modulation} \label{subsec:random-modu}

Consider the RMod model in \ref{eq:cycl-constr-basic}. Recall from Section \ref{sec:intro} that the bivariate stationary series $Y = \{Y_{j,n}\}_{j=1,2,n \in \ZZ}$ should satisfy property \eqref{Property1} for the RMod model to be stationary. In this section, we find one convenient representation of the series $Y$ that will satisfy \eqref{Property1}. More specifically, consider the two-sided linear representation \eqref{eq:Zn}, where $A_l \in \RR^{2 \times 2}$ are such that $\sum_{l=-\infty}^\infty \|A_l \|_F^2 < \infty$, where $\| \cdot \|_F$ denotes the Frobenius norm of a matrix.
\begin{proposition} \label{thm:linear-repr-biv}
   Let a bivariate second-order stationary series $Y = \{Y_n\}_{n \in \ZZ}$ be given by \eqref{eq:Zn}. If 
    \begin{equation} \label{eq:A+A-general}
    A_l =
      \begin{pmatrix}
          a_{0,l} &  a_{1,l} \\
          -a_{1,l} &  a_{0,l}
      \end{pmatrix}, \quad  l \in \ZZ, 
     \end{equation}
then $Y$ satisfies property \eqref{Property1}.
\end{proposition}

\begin{proof}
Recall that $\EE \veps_n \veps_n^T = I_2$. Then, 
\begin{equation}
\begin{split}
    \gamma_Y(h) &=\EE Y_h Y_0^T  = \sum_{l = -\infty}^\infty A_{l + h} A_l^T = \sum_{l = -\infty}^{\infty} \begin{pmatrix}
        a_{0,l+h}  & a_{1,l+h} \\
        -a_{1,l+h} & a_{0,l+h}
    \end{pmatrix}  \begin{pmatrix}
        a_{0,l}  & -a_{1,l} \\
        a_{1,l} & a_{0,l}
    \end{pmatrix} \\
    &= \sum_{l = -\infty}^{\infty} \begin{pmatrix}
        a_{0,l+h} a_{0,l} + a_{1,l+h} a_{1,l} & -a_{0,l+h} a_{1,l} + a_{1,l+h} a_{0,l} \\
         -a_{1,l+h} a_{0,l} + a_{0,l+h} a_{1,l}  &  a_{0,l+h} a_{0,l} + a_{1,l+h} a_{1,l}
    \end{pmatrix}.
\end{split}
\end{equation}
This says that $\gamma_Y$ satisfies the required property.
\end{proof}

\begin{remark}
    We conjecture that any second-order stationary series $Y$ satisfying property \eqref{Property1} has a linear representation \eqref{eq:Zn} with coefficient matrices given by \eqref{eq:A+A-general}.
\end{remark}

Property \eqref{Property1} and the resulting ACVF in \eqref{eq:acvf-X-model} are expressed in the time domain. In the following result, we express them in the spectral domain.

\begin{proposition} \label{prop:spectral-property}
    Suppose the matrix ACVF $\gamma_Y$ of a bivariate series $Y$ satisfies property \eqref{Property1}. Assume that $Y$ has the spectral density $f_Y(\lambda) = (f_{Y,jk}(\lambda))_{j,k=1,2}$ connected to $\gamma_Y$ through the usual relations: $
        f_Y(\lambda) = \frac{1}{2\pi} \sum_{h=-\infty}^\infty e^{-i\lambda h} \gamma_Y(h), \; \gamma_Y(h) = \int_{-\pi}^\pi e^{ih\lambda} f(\lambda) d\lambda.$
    Then, property \eqref{Property1} is equivalent to:
    \begin{equation} \label{eq:property-spec} \tag{P-S}
        f_{Y,11}(\lambda) = f_{Y,22}(\lambda)\quad \text{and}\; \quad f_{Y,12}(\lambda) = - f_{Y,21}(\lambda), \quad \text{for all } \lambda \in (-\pi,\pi).
    \end{equation}
    Furthermore, under \eqref{Property1}, the spectral density of the RMod series \eqref{eq:cycl-constr-basic} can be expressed as
    \begin{equation} \label{eq:f-X-from-Y-general}
            f_X(\lambda) = \frac{1}{2} \left[ f_{Y,11}(\lambda - \lambda_0) + f_{Y,11}(\lambda + \lambda_0)    \right] - \frac{1}{2i} \left[  f_{Y,12}(\lambda - \lambda_0) - f_{Y,12}(\lambda + \lambda_0)  \right].
    \end{equation}
\end{proposition}

\begin{proof}
The claim in \eqref{eq:property-spec} follows upon writing, by \eqref{Property1},
\[
f_Y(\lambda) = \frac{1}{2\pi} \sum_{h=-\infty}^\infty e^{-i\lambda h} \gamma_Y(h) =  \frac{1}{2\pi} \sum_{h=-\infty}^\infty e^{-i\lambda h} \begin{pmatrix}
    \gamma_{Y,11}(h) & \gamma_{Y,12}(h) \\
    -\gamma_{Y,12}(h) & \gamma_{Y,11}(h)
\end{pmatrix}.
\]
For the claim in \eqref{eq:f-X-from-Y-general}, note that
    \begin{equation}
    \begin{split}
        f_X(\lambda) &= \frac{1}{2\pi} \sum_{h=-\infty}^\infty e^{-ih\lambda} \gamma_X(h) \\
        &=\frac{1}{2\pi} \sum_{h=-\infty}^\infty e^{-ih\lambda} \left( \cos(\lambda_0 h) \gamma_{Y,11}(h) - \sin(\lambda_0 h) \gamma_{Y,12}(h)   \right) \\
        &= \frac{1}{2} \frac{1}{2\pi} \sum_{h=-\infty}^\infty \left[ e^{-ih(\lambda-\lambda_0)}  \gamma_{Y,11}(h) + e^{-ih(\lambda+\lambda_0)}  \gamma_{Y,11}(h) \right] \\
        &\quad - \frac{1}{2i} \frac{1}{2\pi} \sum_{h=-\infty}^\infty \left[ e^{-ih(\lambda-\lambda_0)}  \gamma_{Y,12}(h) - e^{-ih(\lambda+\lambda_0)}  \gamma_{Y,12}(h)
        \right] \\
        &= \frac{1}{2} \left[ f_{Y,11}(\lambda - \lambda_0) + f_{Y,11}(\lambda + \lambda_0)    \right] - \frac{1}{2i} \left[  f_{Y,12}(\lambda - \lambda_0) - f_{Y,12}(\lambda + \lambda_0)  \right],
    \end{split}
    \end{equation}
where we have used the formula for $\gamma_X$ from \eqref{eq:acvf-X-model} in the second line and the representations of $\cos$ and $\sin$ as complex exponentials in the third. 
\end{proof}

We next provide a linear representation for the RMod series $\{X_n\}_{n \in \ZZ}$ in \eqref{eq:cycl-constr-basic} when the bivariate series $\{Y_n\}_{n \in \ZZ}$ assumes the form in \eqref{eq:Zn}--\eqref{eq:A+A-general}. Although a special case of the following result was proved in \cite{ProMad22}, we have not encountered an analogous result in the RMod literature. Recall the definition of the sign function
\begin{equation} \label{def:sign-func}
    \text{sign}(c) = \begin{cases}
    1 & c > 0, \\
    0 & c = 0, \\
    -1 & c < 0.
\end{cases}
\end{equation}

\begin{theorem}
\label{thm:cycl-linear-general}
    Let $\{X_n\}_{n \in \ZZ}$ and $\{Y_n\}_{n \in \ZZ}$ be given by \eqref{eq:cycl-constr-basic} and \eqref{eq:Zn}--\eqref{eq:A+A-general} respectively. Then, $\{X_n\}_{n \in \ZZ}$ has the linear representation
    \begin{equation} \label{eq:cycl-linear-general}
        X_n = \sum_{j = -\infty}^\infty  \begin{pmatrix}
    \cos(\lambda_0 j) a_{0,j} - \sin(\lambda_0 j) a_{1,j} \\
    \cos(\lambda_0 j) a_{1,j} + \sin(\lambda_0 j) a_{0,j}
\end{pmatrix}^T  \tilde \veps_{n-j} , \quad n \in \ZZ,
     \end{equation}
    where $\{\tilde \veps_n\}_{n \in \ZZ}$ is a WN sequence defined as, for $n \in \ZZ$,
\begin{equation}
    \tilde \veps_n \doteq \begin{pmatrix}
        \cos(\lambda_0 n) & \operatorname{sign}(n) \sin(\lambda_0 n) \\
        - \operatorname{sign}(n) \sin(\lambda_0 n) & \cos(\lambda_0 n)
    \end{pmatrix} \veps_n.
\end{equation}
In particular, $\EE (\tilde \veps_k \tilde \veps_l^T) = I_2 \delta_{kl}$, where $\delta_{kl}$ denotes the Kronecker delta.
\end{theorem}

\begin{proof}
    Define 
    \[
    M \doteq \begin{pmatrix}
        \cos(\lambda_0) & \sin(\lambda_0) \\
        -\sin(\lambda_0) & \cos(\lambda_0)
    \end{pmatrix}, \quad 
    v_n \doteq \begin{pmatrix}
        \cos(\lambda_0 n) & \sin(\lambda_0 n)
    \end{pmatrix}^T,   \quad
    v_0 \doteq \begin{pmatrix}
        1 & 0
    \end{pmatrix}^T,
    \]
    and observe that the following identities hold
    \[
    M^j =
    \begin{cases}
        \begin{pmatrix}
        \cos(\lambda_0 j) & \sin(\lambda_0 j) \\
        -\sin(\lambda_0 j) & \cos(\lambda_0 j) 
    \end{pmatrix} & \text{if }j \ge 1, \vspace{.3cm}\\
        \begin{pmatrix}
        \cos(\lambda_0 j) & -\sin(\lambda_0 j) \\
        \sin(\lambda_0 j) & \cos(\lambda_0 j) 
    \end{pmatrix} & \text{if }j \le 0.
   \end{cases}
   \quad
   v_j^T =   v_0^T M^j, \quad j \in \ZZ.
    \]
Recall the linear representation of $Y_n$ in \eqref{eq:A+A-general}, and write
\begin{equation}
X_n = v_n^T Y_n = v_0^T M^n \sum_{j= -\infty}^\infty A_j \veps_{n-j} \\ 
        = \sum_{j=-\infty}^\infty v_0^T M^j M^{n-j} A_j \veps_{n-j}.
\end{equation}
Since for all $k \in \ZZ$, $M^k = \begin{pmatrix}
    \alpha & \beta \\
    -\beta & \alpha
\end{pmatrix}$ for some $\alpha,\beta$, the following commutativity property holds:
\[
M^k A_l = A_l M^k, 
\]
for all $l,k \in \ZZ$. Therefore,
    \begin{equation}
    \begin{split}
        X_n  = \sum_{j=-\infty}^\infty v_0^T M^j M^{n-j} A_j \veps_{n-j} 
        = \sum_{j=-\infty}^\infty v_0^T M^j A_j M^{n-j} \veps_{n-j} 
        = \sum_{j=-\infty}^\infty v_j^T  A_j \tilde \veps_{n-j} = \sum_{j=-\infty}^\infty \tilde A_j \tilde \veps_{n-j},
    \end{split}
    \end{equation}
where, for $k \in \ZZ$,
\[
\tilde A_k \doteq v_k^T  A_k = \begin{pmatrix}
    \cos(\lambda_0 k) a_{0,k} - \sin(\lambda_0 k) a_{1,k} \\
    \cos(\lambda_0 k) a_{1,k} + \sin(\lambda_0 k) a_{0,k}
\end{pmatrix}^T, \quad \tilde \veps_{k} \doteq M^{k}  \veps_{k}.
\]
In particular, upon noting that, for all $k \in \ZZ$,
$M^k (M^k)^T = I_2$, we see that
\[
 \EE\left( \tilde \veps_{k} \tilde \veps_{k}^T   \right) = \EE \left( M^{k}  \veps_{k} \veps_{k}^T  (M^k)^T  \right) = M^k (M^k)^T =  I_2.
\]
This concludes the proof.
\end{proof}

\section{CLM Model From Random Modulation} \label{sec:cycl-constru}

\subsection{CLM Model Construction} \label{subsec:cycl-constr}

In view of Proposition \ref{thm:linear-repr-biv}, its condition \eqref{eq:A+A-general}, and the representation \eqref{eq:Zn}--\eqref{eq:Al-asympt}, it is now straightforward to construct bivariate LM series $Y$ satisfying property \eqref{Property1}. From property \eqref{Property1}, note that for this $Y$, the limiting form of the ACVF in \eqref{eq:biv-gamma-general} necessarily has
\begin{align} \label{eq:d1d2d}
    d_1 = &d_2 \doteq d,  \\
 \begin{pmatrix}  
        R_{11} & R_{12} \\
        R_{21} & R_{22}
    \end{pmatrix} &\doteq 
    \begin{pmatrix}
        r_0 & -r_1 \\
        r_1 & r_0
    \end{pmatrix}     \label{eq:acvf-asympotitc-Y} 
\end{align}
for some $r_0 > 0,r_1 \in \RR$. In view of \eqref{eq:acvf-asympotitc-Y} and \eqref{eq:property-spec} in Proposition \ref{prop:spectral-property}, the limiting form of the spectral density in \eqref{eq:biv-spectral-general} necessarily has
\begin{equation} \label{eq:spectral-asympotitc-Y}
       \begin{pmatrix}
        g_{11} & g_{12} e^{-i \om} \\
        g_{12} e^{i \om} & g_{22}
    \end{pmatrix} = 
    \begin{pmatrix}
        g_0 & i g_1 \\
        -i g_1 & g_0
    \end{pmatrix},
\end{equation}
for some $g_0,g_1 \in \RR, g_0 \neq 0,$ that is, a very special bivariate phase parameter $\om = - \frac{\pi}{2}$. (Equivalently, we can set $\om = \frac{\pi}{2}$, which would change the sign of $g_1$.)

Consider now a bivariate LM series $Y$ with the representation \eqref{eq:Zn}--\eqref{eq:Al-asympt}. For this $Y$ to satisfy property \eqref{Property1}, Proposition \ref{thm:linear-repr-biv} requires $A_l$ to have the form \eqref{eq:A+A-general}. This implies taking \eqref{eq:d1d2d} and also $A_{11}^{\pm} = A_{22}^{\pm}$, $A_{12}^{\pm} = - A_{21}^{\pm}$, in \eqref{eq:Al-asympt}. The memory parameter $d$ aside, this formulation has 4 parameters (2 for + and 2 for --), compared to 2 parameters in \eqref{eq:acvf-asympotitc-Y} or \eqref{eq:spectral-asympotitc-Y}. To reduce the number of parameters to 2, we suggest imposing
\begin{equation} \label{eq:A+A-format}
    A^+ \doteq   \begin{pmatrix}
        A_{11}^+ & A_{12}^+ \\
        A_{21}^+ & A_{22}^+
    \end{pmatrix} =
    \begin{pmatrix}
        a_0  & a_1 \\
        -a_1 & a_0
    \end{pmatrix}  
      =  \begin{pmatrix}
        A_{11}^- & A_{12}^- \\
        A_{21}^- & A_{22}^-
    \end{pmatrix}^T   \doteq (A^-)^T
\end{equation}
for some $a_0,a_1 \in \RR$. The transpose in \eqref{eq:A+A-format} is important, e.g., without the transpose, the resulting model necessarily leads to $R_{12} = R_{21} = 0$.

The following two results relate the parameters in the various forms \eqref{eq:d1d2d}--\eqref{eq:A+A-format} above, and describe the CLM structure of the RMod series \eqref{eq:cycl-constr-basic} resulting from the bivariate LM series $Y$ chosen above. The fact that the RMod $X$ in \eqref{eq:cycl-constr-basic} will have CLM follows immediately from the construction above, since by using \eqref{eq:acvf-X-model}, \eqref{eq:biv-gamma-general}, and \eqref{eq:acvf-asympotitc-Y},
\begin{align} \label{eq:rmod-basic}
  \gamma_X(h) &= \gamma_{Y,11}(h) \cos(\lambda_0 h) - \gamma_{Y,12}(h) \sin(\lambda_0h) \\
    &\simeq   R_{11} \cos(\lambda_0 h)  h^{2d-1} - R_{12} \sin(\lambda_0 h)       h^{2d-1} \\
&=  \left[ \cos(\lambda_0 h) r_0  + \sin(\lambda_0 h) r_1      \right]h^{2d-1} \\
    &= c_\gamma \cos(\lambda_0 h + \phi) h^{2d-1}, \quad \text{as } h \to \infty \label{eq:acvf-X-generic},
\end{align}
for $c_\gamma,\phi$ specified in the second result below.

\begin{proposition} \label{prop:r0-r1-a0-a1-g0-g1}
Let $Y$ have the representation \eqref{eq:Zn}--\eqref{eq:Al-asympt} with $A_l$ in \eqref{eq:Al-asympt} satisfying \eqref{eq:A+A-general} and \eqref{eq:A+A-format}. Let also the limiting form of the ACVF of $Y$ (resp. spectral density) be given by \eqref{eq:acvf-asympotitc-Y} (resp. \eqref{eq:spectral-asympotitc-Y}). Then,
\begin{equation} \label{eq:r0-r1-a0-a1}
r_0 = \frac{\Gamma^2(d)}{\Gamma(2d)} \left[   (a_0^2 + a_1^2) \frac{1}{\cos(\pi d)} + (a_0^2 - a_1^2) \right], \quad
r_1 = - 2 a_0 a_1 \frac{\Gamma^2(d)}{\Gamma(2d)} 
\end{equation}
and
    \begin{equation} \label{eq:g0-g1-r0-r1}
    g_0 = \frac{\Gamma(2d) }{ \pi }  \cos(\pi d) r_0, \quad g_1 = \frac{\Gamma(2d)}{\pi} \sin(\pi d)  r_1.
    \end{equation}
\end{proposition}

\begin{proof}
The first part of \eqref{eq:r0-r1-a0-a1} follows from Proposition 3.1 in \cite{KecPi15}. To apply that proposition, we compute
\begin{equation*}
\begin{split}
    C^1 &= A^- (A^-)^T = \begin{pmatrix}
        a_0 & - a_1 \\
        a_1 & a_0 
    \end{pmatrix} \begin{pmatrix}
        a_0 &  a_1 \\
        - a_1 & a_0 
    \end{pmatrix} =
    \begin{pmatrix}
        a_0^2 + a_1^2 & 0\\
        0 & a_0^2 + a_1^2
    \end{pmatrix}, \\
    C^2 &= A^- (A^+)^T = \begin{pmatrix}
        a_0 & - a_1 \\
        a_1 & a_0 
    \end{pmatrix} \begin{pmatrix}
        a_0 &  -a_1 \\
        a_1 & a_0 
    \end{pmatrix} =
    \begin{pmatrix}
        a_0^2 - a_1^2 & -2 a_0 a_1 \\
        2 a_0 a_1  & a_0^2 - a_1^2
    \end{pmatrix}, \\
      C^3 &= A^+ (A^+)^T = \begin{pmatrix}
        a_0 & a_1 \\
        -a_1 & a_0 
    \end{pmatrix} \begin{pmatrix}
        a_0 &  -a_1 \\
        a_1 & a_0 
    \end{pmatrix} =
    \begin{pmatrix}
        a_0^2 + a_1^2 & 0\\
        0 & a_0^2 + a_1^2
    \end{pmatrix}.
\end{split}
\end{equation*}
This leads to
\begin{equation*}
\begin{split}
R_{11} = R_{22} &= \frac{\Gamma^2(d)}{\Gamma(2d)} \left[  2 (a_0^2 + a_1^2) \frac{\sin(\pi d)}{\sin(2 \pi d)} + (a_0^2 - a_1^2) \right] = \frac{\Gamma^2(d)}{\Gamma(2d)} \left[  (a_0^2 + a_1^2) \frac{1}{\cos(\pi d)} + (a_0^2 - a_1^2) \right], \\
R_{12} &= 2 a_0 a_1 \frac{\Gamma^2(d)}{\Gamma(2d)} ,
\end{split}
\end{equation*}
confirming \eqref{eq:r0-r1-a0-a1}, since $R_{11} = r_0$ and $R_{12} = -r_1$. Note that in the above, we have calculated $R_{12}$ from the formula for $R_{21}$ in Proposition 3.1 of \cite{KecPi15}, to account for the different definition of $\gamma_Y(h) = \EE Y_h Y_0^T $ that is used in this work, compared to $\gamma_Y(h) = \EE Y_0 Y_h^T$ in \cite{KecPi15}.

The second part \eqref{eq:g0-g1-r0-r1} follows from Proposition 2.1 in \cite{KecPi15}. In particular, since $R_{11} = R_{22} = r_0$ and $R_{12} = -R_{21} = -r_1$, we obtain that 
\[
    g_0 = g_{11} = g_{22} = \frac{\Gamma(2d)}{\pi}  r_0 \cos(\pi d),
\]
    and
    \[
    i g_1 = g_{12} e^{-i\om} = \frac{\Gamma(2d)}{2\pi}\left( -i(-r_1 - r_1) \sin(\pi d) \right),
    \]
    which yields
    
    \hfill $g_1 =   \frac{\Gamma(2d)}{\pi} r_1 \sin(\pi d)$   \hfill
\end{proof}
\begin{remark}
    The relations in \eqref{eq:r0-r1-a0-a1} can be solved for $a_0,a_1$. By solving for $a_0$ in the second equation in \eqref{eq:r0-r1-a0-a1} and plugging this into the first equation, we obtain a fourth order polynomial in $a_1$ which, by setting $\nu = a_1^2$, reduces to a quadratic of the form
    \[
    \alpha \nu^2 - r_0 \nu + \gamma r_1^2 = 0,
    \]
    where
    \begin{equation} \label{eq:alpha-def}
        \alpha \doteq \frac{\Gamma^2(d)}{\Gamma(2d)} \left( \frac{1}{\cos(\pi d)} - 1 \right), \; \gamma \doteq \frac{\Gamma(2d)}{4\Gamma^2(d)} \left( \frac{1}{\cos(\pi d)} + 1 \right).
    \end{equation}
    This leads to the discriminant
    \begin{equation} \label{eq:discriminant}
        \Delta \doteq r^2_0 - 4 \alpha \gamma r_1^2 = r_0^2 -  \left(  \frac{1}{\cos^2(\pi d)} - 1  \right) r_1^2 = r_0^2 -  \tan^2(\pi d) r_1^2.
     \end{equation}
    Then, for the values of $r_0 , r_1$ for which $\Delta \ge 0$ (i.e., $r_0 \in [-\tan(\pi d) r_1,\tan(\pi d) r_1]$), we can solve the quadratic and obtain the solutions (possibly one) $\nu_{\pm} = \frac{r_0 \pm \sqrt{\Delta}}{2\alpha}$. Since $r_0 > 0$, we have $\nu_+ \ge \nu_- \ge 0$. The solution $\nu_+ = a_1^2$ leads to
    \begin{equation} \label{eq:a0a1-r0r1}
    \begin{split}
            a_{1,+} &= \sqrt{\frac{r_0 + \sqrt{\Delta}}{2\alpha}} = \sqrt{\frac{r_0 + \sqrt{r_0^2 -  \tan^2(\pi d) r_1^2}}{2\frac{\Gamma^2(d)}{\Gamma(2d)} \left( \frac{1}{\cos(\pi d)} - 1 \right)}}, \\
            a_{0,+} &= -\frac{\sqrt{2\alpha} r_1 \Gamma(2d)}{2 \sqrt{r_0 + \sqrt{\Delta}} \Gamma^2(d)} = -\frac{\sqrt{2\Gamma(2d) \left( \frac{1}{\cos(\pi d)} - 1 \right)} r_1 }{2 \sqrt{r_0 + \sqrt{ r_0^2 -  \tan^2(\pi d) r_1^2}} \Gamma(d)}.
     \end{split}
    \end{equation}
    Analogous calculations give $a_{1,-},a_{0,-}$ for the solution $\nu_- = a_1^2$. From now on, we write $a_0$ (resp., $a_1$) to mean $a_{0,+}$ (resp., $a_{1,+}$) for simplicity.
\end{remark}

\begin{proposition} \label{thm:constr}
    Let $\{Y_n\}_{n \in \ZZ}$ be a stationary bivariate LM series satisfying \eqref{Property1} with limiting ACVF satisfying \eqref{eq:acvf-asympotitc-Y}. Let also $\{X_n\}_{n \in \ZZ}$ be the RMod series \eqref{eq:cycl-constr-basic} for some $\lambda_0 \in (0,\pi)$. Then, $\{X_n\}_{n \in \ZZ}$ is second-order stationary, its ACVF satisfies \eqref{eq:acvf-X-model}, and exhibits CLM with the same $\lambda_0 \in (0,\pi)$ and
    \begin{equation} \label{eq:phi-def}
c_\gamma \doteq \sqrt{r_0^2 + r_1^2}, \quad    \phi \doteq \arcsin \left(\frac{-r_1}{\sqrt{r_0^2 + r_1^2}}\right).
\end{equation}
Equivalently, $r_0 = c_\gamma \cos(\phi), r_1 = - c_\gamma \sin(\phi)$.
\end{proposition}


\begin{proof}
For $h,n \in \ZZ$, we have by \eqref{eq:cycl-constr-basic},
\begin{equation} \label{eq:Y-ACVF}
\begin{split}
     \EE X_{n+h} X_{n}  &= \cos(\lambda_0 n) \cos( \lambda_0 (n+h)) \gamma_{Y,11}(h) + \sin(\lambda_0 n) \sin( \lambda_0 (n+h)) \gamma_{Y,22}(h) \\
    &\quad+ \cos(\lambda_0 n) \sin( \lambda_0 (n+h)) \gamma_{Y,21}(h) + \sin(\lambda_0 n) \cos( \lambda_0 (n+h)) \gamma_{Y,12}(h) \\
    &= \gamma_{11}(h) \bigg( \cos(\lambda_0 n) \cos( \lambda_0 (n+h)) + \sin(\lambda_0 n) \sin( \lambda_0 (n+h))   \bigg) \\
    &\quad+ \gamma_{12}(h) \bigg(- \cos(\lambda_0 n) \sin( \lambda_0 (n+h)) + \sin(\lambda_0 n) \cos( \lambda_0 (n+h)) \bigg),
\end{split}
\end{equation}
by using \eqref{Property1}. This implies that
\begin{equation} \label{eq:acvf-X-Y}
\begin{split}
    \gamma_X(h) 
    &= \gamma_{11}(h) \cos(\lambda_0 h) - \gamma_{12}(h) \sin(\lambda_0h) \\
    &\simeq \left[ r_0 \cos(\lambda_0 h ) + r_1 \sin(\lambda_0 h)  \right] h^{2d-1} \\
    &= \sqrt{r_0^2 + r_1^2} \left[ \frac{r_0}{\sqrt{r_0^2 + r_1^2}} \cos(\lambda_0 h) + \frac{r_1}{\sqrt{r_0^2 + r_1^2}} \sin(\lambda_0 h)       \right]h^{2d-1} \\
&= c_\gamma \left[ \cos(\phi) \cos(\lambda_0 h) - \sin(\phi) \sin(\lambda_0 h)       \right]h^{2d-1} \\
    &= c_\gamma \cos(\lambda_0 h + \phi) h^{2d-1},
\end{split}
\end{equation}
where $c_\gamma,\phi$ are defined in \eqref{eq:phi-def}. It thus exhibits CLM with parameters $c_\gamma,\phi$ given in \eqref{eq:phi-def}.
\end{proof}

\begin{remark}\label{r:adm-phi}
    Note that the expression for $\phi$ in \eqref{eq:phi-def} both leads to the admissible set \eqref{eq:admiss-sets-def} and shows that any cyclical phase in the set can be achieved with a suitable choice of $r_0,r_1$. To see these points, substitute $r_0,r_1$ in \eqref{eq:phi-def} by $g_0,g_1$ using the relations in \eqref{eq:g0-g1-r0-r1}. This leads to
    \[
    \phi = \arcsin \left( \frac{-\pi g_1}{\Gamma(2d) \sin(\pi d)} \frac{1}{\sqrt{\frac{\pi^2 g_0^2}{\Gamma(2d)^2 \cos^2(\pi d)} + \frac{\pi^2 g_1^2}{\Gamma^2(2d) \sin^2(\pi d)}}} \right) = \arcsin \left( \frac{-\operatorname{sign}(g_1)}{\sqrt{\left(\frac{g_0}{g_1}\right)^2 \tan^2(\pi d)+1}}  \right).
    \]
    As $|g_1 | \le g_0$ in the spectral domain, the cyclical phase takes its values in any point of the interval
    \[
    \left[ \arcsin\left(-\frac{1}{\sqrt{1+\tan^2(\pi d)}} \right), \arcsin\left(\frac{1}{\sqrt{1+\tan^2(\pi d)}} \right)  \right] = \left[ \arcsin(-\cos(\pi d)), \arcsin( \cos(\pi d))  \right] = \cli_d,
    \]
    where $\cli_d$ is the admissible set in \eqref{eq:admiss-sets-def}. For the boundary points, we have: 
    \begin{equation} \label{eq:borderline-phi-g-r}
        \phi = \pm \left( \frac{1}{2} - d \right) \pi \Leftrightarrow g_1 = \mp g_0 \Leftrightarrow r_0 = \mp \tan(\pi d) r_1.
     \end{equation}
    The relation \eqref{eq:a0a1-r0r1} shows that any admissible phase can be obtained with a suitable choice of $a_0,a_1$. For given $c_\gamma > 0$ and $\phi \in \cli_d$, we have, from Proposition \ref{thm:constr} and \eqref{eq:a0a1-r0r1} that one choice of $a_0,a_1$ leading to $c_\gamma$ and $\phi$ is
    \[
            a_1 = \sqrt{\frac{c_\gamma \cos(\phi) + c_\gamma \sqrt{\cos^2(\phi)-  \tan^2(\pi d) \sin^2(\phi)}}{2\frac{\Gamma^2(d)}{\Gamma(2d)} \left( \frac{1}{\cos(\pi d)} - 1 \right)}}, \quad
            a_0 = -\frac{\sqrt{2 c_\gamma \Gamma(2d) \left( \frac{1}{\cos(\pi d)} - 1 \right)}  \sin(\phi)}{2 \Gamma(d) \sqrt{ \cos(\phi) +  \sqrt{ \cos^2(\phi)-  \tan^2(\pi d) \sin^2(\phi)}} }.
    \]
    In particular, the ``boundary" case $\phi = \pm \left( \frac{1}{2} - d \right) \pi$ in \eqref{eq:borderline-phi-g-r} corresponds to
    \begin{equation} \label{eq:boundary-a0a1}
        \Delta = r^2_1 \left( \frac{\sin^2(\pi d)-1}{\cos^2(\pi d)} + 1 \right) = 0,\quad a_{1,\pm} = \sqrt{\frac{r_0}{2\alpha}}, \quad a_{0,\pm} = \mp \frac{\sqrt{2\alpha} r_1 \Gamma(2d)}{2 \sqrt{r_0} \Gamma^2(d)} = \pm \frac{\sqrt{2\alpha r_0} \Gamma(2d)}{\tan(\pi d) \Gamma^2(d)},
    \end{equation}
    where $\alpha$ was defined in \eqref{eq:alpha-def} and $\Delta$ in \eqref{eq:discriminant}. Note that the $\operatorname{sign}$ of $\phi$ is only differentiated by the form of $a_0$.
    
\end{remark}

\subsection{Parametric Models}
\label{subsec:bivar-constr}

\subsubsection{Capturing CLM effects} \label{subsubsec:clm-effects}

We shall now provide a parametric model for the bivariate LM series considered in Section \ref{subsec:cycl-constr}, that can be used in the RMod construction of CLM series with general cyclical phase. Let $B$ be the usual backward shift operator acting as $B^k Z_n = Z_{n-k}, k \in \ZZ$, e.g., $B^{-1} Z_n = Z_{n+1}$. Set
\begin{equation} \label{eq:D-def-diag}
    D = \text{diag}(d,d) = \begin{pmatrix}
        d & 0 \\
        0 & d
    \end{pmatrix},
\end{equation}
and define a bivariate fractional integration operator $(I- B^{\pm 1})^{-D} = \text{diag} ((I- B^{\pm 1})^{-d},(I- B^{\pm 1})^{-d}),$ where $(I- B^{\pm 1})^{-d} = \sum_{k=0}^{\infty} c_{d,k} B^{\pm k}$ and $c_{d,k}$ are the coefficients in the Taylor expansion
\begin{equation} \label{eq:Taylor-exp}
    (1-x)^{-d} = \sum_{k=0}^\infty c_{d,k} x^k.
\end{equation}
The coefficients $c_{d,k}$ satisfy
\begin{equation} \label{eq:cdk-asympt}
   c_{d,k} \doteq \frac{\Gamma(k+d)}{\Gamma(k+1)\Gamma(d)}\sim \frac{k^{d-1}}{\Gamma(d)}, \quad \text{as }k \to + \infty 
\end{equation}
(e.g., equations (2.4.2) and (2.4.5) in \cite{pipiras_taqqu_2017}).

Take a bivariate WN series $\{\veps_n\}_{n \in \ZZ}$ with the covariance matrix $\EE \veps_n \veps_n^T = I_2$. For the parametric bivariate LM series $Y$, set
\begin{equation} \label{eq:Yn}
    Y_n = (I-B)^{-D} Q_+ \veps_n + (I-B^{-1})^{-D} Q_- \veps_n,
\end{equation}
where
\begin{equation} \label{eq:def-Q+-}
    Q_+ \doteq  \begin{pmatrix}
        q_0 & q_1 \\
       - q_1 & q_0
    \end{pmatrix} \doteq Q_-^T.
\end{equation}
The series $\{Y_n\}$ in \eqref{eq:Yn} is an example of bivariate FARIMA($0,D,0$) series considered in \cite{KecPi15}. By construction, note that the series $Y$ in \eqref{eq:Yn}--\eqref{eq:def-Q+-} has the form of the series $Y$ considered in Section 3.1 with
\begin{equation} \label{eq:parametric-to-a0a1}
    a_0 = \frac{q_0}{\Gamma(d)}, \quad a_1 = \frac{q_1}{\Gamma(d)},
\end{equation}
where we used \eqref{eq:cdk-asympt}. The series $Y$ in \eqref{eq:Yn}--\eqref{eq:def-Q+-} thus satisfies Property \eqref{Property1}. Relations similar to those in Propositions \ref{prop:r0-r1-a0-a1-g0-g1} and \ref{thm:constr} but in terms of $q_0,q_1$ follow by using \eqref{eq:parametric-to-a0a1}. In the following results, we also derive the explicit expressions for the ACVF and spectral density of the parametric model \eqref{eq:Yn}--\eqref{eq:def-Q+-}.

\begin{proposition} \label{thm:Yn-constr}
 Let $Y$ be given by \eqref{eq:Yn}--\eqref{eq:def-Q+-}. Then, $\{Y_n\}_{n \in \ZZ}$ satisfies \eqref{Property1} with, for $h \in \ZZ$,
 \begin{equation} \label{eq:acvf-parametric-finite-lag}
 \begin{split}
     \gamma_{Y,11}(h) &= \gamma_{Y,22}(h) \\
     &= q_0^2 \left( 2 \frac{\Gamma(1-2d) \Gamma(h+d) \sin(\pi d)}{ \Gamma(h+1-d) \pi}+ \frac{\Gamma(2d+h)}{\Gamma(2d) \Gamma(1+h)}  \1_{\{h \ge 0\}} +  \frac{\Gamma(2d-h)}{\Gamma(2d) \Gamma(1-h)} \1_{\{h \le 0\}} \right)  \\
    &\quad+q_1^2 
    \left( 2 \frac{\Gamma(1-2d) \Gamma(h+d) \sin(\pi d)}{ \Gamma(h+1-d) \pi} - \frac{\Gamma(2d+h)}{\Gamma(2d) \Gamma(1+h)}  \1_{\{h \ge 0\}} - \frac{\Gamma(2d-h)}{\Gamma(2d) \Gamma(1-h)} \1_{\{h \le 0\}}  \right),\\
    \gamma_{Y,12}(h) &= - \gamma_{Y,21}(h) = 
        \operatorname{sign}(h)  \frac{2 q_0 q_1}{ \Gamma(2d)} \frac{\Gamma(h+2d)}{\Gamma(1+h)}, 
  \end{split}
 \end{equation}
and has limiting ACVF as in \eqref{eq:acvf-asympotitc-Y} with
    \begin{equation} \label{eq:def-a-b}
     r_0 \doteq \frac{1}{\Gamma(2d)} \left[ q_0^2 \left( \frac{1}{\cos(\pi d)} + 1  \right) + q_1^2 \left( \frac{1}{\cos(\pi d)} - 1  \right) \right] , \quad r_1 \doteq - \frac{2q_0 q_1}{\Gamma(2d)}.
    \end{equation}
\end{proposition}

\begin{proof}
The series $\{Y_n\}_{n \in \ZZ}$ satisfies property \eqref{Property1} by Theorem \ref{thm:linear-repr-biv}. We next obtain tractable formulas relating the ACVF of $Y$ with the parameters $q_0,q_1$. Consider
\begin{equation} \label{eq:covar-Y-b}
\begin{split}
b^1 &\doteq Q_- Q_-^T =
\begin{pmatrix}
q_0^2+q_1^2 & 0 \\
0 & q_0^2+q_1^2
\end{pmatrix} , \quad 
b^2 \doteq Q_- Q_+^T =
\begin{pmatrix}
q_0^2-q_1^2 & -2q_0 q_1  \\
2q_0 q_1 & q_0^2-q_1^2
\end{pmatrix},
\\
b^3 &\doteq Q_+ Q_+^T =
\begin{pmatrix}
q_0^2+q_1^2 & 0 \\
0 & q_0^2+q_1^2
\end{pmatrix}, \quad
b^4 \doteq Q_+ Q_-^T =
\begin{pmatrix}
q_0^2-q_1^2 & 2q_0 q_1  \\
-2q_0 q_1 & q_0^2-q_1^2
\end{pmatrix}.
\end{split}
\end{equation}
Observe that, in view of \eqref{eq:D-def-diag}, we have that $\gamma_{1,kj} (h) = \gamma_{3,kj}(h)$ and $\gamma_{4,kj} (h) = \gamma_{2,kj}(-h),k,j=1,2$ for all $h \in \ZZ $, where $\gamma_{l,kj},l=1,\dots,4$ and $k,j=1,2$ are defined in the relations (72), Proposition 5.1 of \cite{KecPi15}. See also Proposition 9.4.3 in \cite{pipiras_taqqu_2017}. Then, from the same proposition,
\begin{equation} \label{eq:gamma11Y}
\begin{split}
    \gamma_{Y,11}(h) &= \frac{1}{2\pi} \left[ b^1_{11} \gamma_{1,11}(h) + b^2_{11} \gamma_{2,11}(h) + b^3_{11} \gamma_{3,11}(h) + b^4_{11} \gamma_{4,11}(h)\right] \\
    &= \frac{1}{2\pi} \left[ 2(q_0^2 + q_1^2) \gamma_{1,11}(h) + (q_0^2 - q_1^2)\left(\gamma_{2,11}(h) \1_{\{h \le 0\}} +  \gamma_{4,11}(h) \1_{\{h \ge 0\}} \right)  \right] \\
    &= q_0^2 \left( 2 \frac{\Gamma(1-2d) \Gamma(h+d) \sin(\pi d)}{ \Gamma(h+1-d) \pi}+ \frac{\Gamma(2d+h)}{\Gamma(2d) \Gamma(1+h)}  \1_{\{h \ge 0\}} + \frac{\Gamma(2d-h)}{\Gamma(2d) \Gamma(1-h)} \1_{\{h \le 0\}} \right)  \\
    &\quad+q_1^2 
    \left( 2 \frac{\Gamma(1-2d) \Gamma(h+d) \sin(\pi d)}{ \Gamma(h+1-d) \pi} - \frac{\Gamma(2d+h)}{\Gamma(2d) \Gamma(1+h)}  \1_{\{h \ge 0\}} - \frac{\Gamma(2d-h)}{\Gamma(2d) \Gamma(1-h)} \1_{\{h \le 0\}}  \right),
\end{split}
\end{equation}
where for the second line we used $b^1_{11} = b^3_{11} = q_0^2+q_1^2$, $b^2_{11} = b^4_{11} = q_0^2-q_1^2$, and the properties for $\gamma$ above, and for the third line we used the form of $\gamma_{1,11},\gamma_{2,11}$. Similarly, on noticing that $b^1_{22} =b^3_{22} =  q_0^2+q_1^2$, $b^2_{22} =b^4_{22} =  q_0^2-q_1^2$, and that $\gamma_{i,11}(h) = \gamma_{i,22}(h)$ for all $h$ and $i=1,2,3,4$, we immediately obtain that
\[
\gamma_{Y,22} (h) = \gamma_{Y,11} (h), \quad \text{for all} \; h \in \ZZ.
\]
Next,
\begin{equation} \label{eq:gamma12Y}
\begin{split}
    \gamma_{Y,12}(h) &= \frac{1}{2\pi} \left[ b^1_{12} \gamma_{1,12}(h) + b^2_{12} \gamma_{2,12}(h) + b^3_{12} \gamma_{3,12}(h) + b^4_{12} \gamma_{4,12}(h)\right] \\
    &= \frac{2q_0q_1}{2 \pi} \left[ - \gamma_{2,12}(h) +  \gamma_{2,12}(-h)  \right] \\
    &= \text{sign}(h)  \frac{2 q_0 q_1}{ \Gamma(2d)} \frac{\Gamma(h+2d)}{\Gamma(1+h)},
\end{split}
\end{equation}
and analogous calculations show that, for each $h \in \ZZ$,
\[
\gamma_{Y,21}(h) = \frac{2q_0q_1}{2 \pi} \left[  \gamma_{2,21}(h) -  \gamma_{2,21}(-h)  \right] = - \gamma_{Y,12}(h).
\]

The following asymptotics follow from \eqref{eq:r0-r1-a0-a1} in Proporistion \ref{prop:r0-r1-a0-a1-g0-g1}: As $h \to + \infty$ (and in particular for $h \ge 0$),
\begin{equation} \label{eq:gamma-asymptotics-first order}
\begin{split}
\gamma_{Y,11}(h) &\sim  \frac{1}{2\pi} \left[ 4 (q_0^2+q_1^2) \Gamma(1-2d) \sin (d\pi)  + 2 (q_0^2-q_1^2)  \frac{\pi}{\Gamma(2d)}   \right]h^{2d-1}, \\
\gamma_{Y,12}(h) &\sim  \frac{2q_0 q_1}{\Gamma(2d)} h^{2d-1}.
\end{split}
\end{equation}
\end{proof}

\begin{proposition} \label{prop:spec-dens-param}
   Let $Y$ be given by \eqref{eq:Yn}--\eqref{eq:def-Q+-}. Then, $\{Y_n\}_{n \in \ZZ}$ satisfies \eqref{eq:property-spec} with, for $\lambda > 0$,
 \begin{equation} \label{eq:param-spectral}
 \begin{split}
  f_{Y,11}(\lambda) &= \frac{2^{-2d} \sin^{-2d} \left( \frac{\lambda}{2}  \right)}{\pi}  \left[ (q_0^2 - q_1^2) \cos\left((\lambda - \pi) d\right) + (q_0^2+q_1^2)   \right], \\
 f_{Y,12}(\lambda) &=  \frac{i}{\pi} \left[  q_0 q_1  2^{1-2d} \sin^{-2d}\left( \frac{\lambda}{2} \right) \sin\left((\lambda - \pi) d \right)   \right],
  \end{split}
 \end{equation}
 and has the limiting spectral decomposition as in \eqref{eq:spectral-asympotitc-Y} with
 \begin{equation} \label{eq:g0-g1-limiting-model}
     g_0 = \frac{1}{\pi} \left[ q_0^2(1 + \cos(\pi d) ) + q_1^2 (1 - \cos(\pi d))    \right], \quad g_1 = - 2\frac{\sin(\pi d)}{\pi}  q_0 q_1.
 \end{equation} 
\end{proposition}

\begin{proof}
    The spectral density matrix is given by
    \[
    f(\lambda) = \frac{1}{2\pi} G(\lambda) G^*(\lambda),
    \]
    where $G(\lambda) = (1 - e^{-i\lambda})^{-D} Q_+ + (1 - e^{i\lambda})^{-D} Q_-$. Then, by denoting $g_j(\lambda)$ the $j$th row of $G(\lambda)$, we can write the $(j,k)$ component as
    \[
    f_{jk}(\lambda) = \frac{1}{2\pi} g_j (\lambda) g^*_k(\lambda).
    \]
    In particular, since $g_1(\lambda) = ( (1-e^{-i\lambda})^{-d}q_0 +(1-e^{i\lambda})^{-d}q_0, (1-e^{-i\lambda})^{-d}q_1 -(1-e^{i\lambda})^{-d}q_1)$, we get that 
    \begin{equation}
    \begin{split}
            f_{11}(\lambda) &= \frac{1}{2\pi} g_1(\lambda) g_1^*(\lambda) \\
            &= \frac{1}{2\pi} \left[ q_0^2 \left( (1-e^{-i\lambda})^{-d}+ (1-e^{i\lambda})^{-d} \right)^2 - q_1^2 \left( (1-e^{i\lambda})^{-d} - (1-e^{-i\lambda})^{-d}  \right)^2  \right] \\
            &= \frac{1}{2\pi} \left[ \left( q_0^2 - q_1^2 \right) (1-e^{-i\lambda})^{-2d} + \left( q_0^2 - q_1^2 \right) (1-e^{i\lambda})^{-2d} + 2 \left( q_0^2 + q_1^2 \right)  (2 - 2 \cos(\lambda))^{-d}  \right] \\
            &= f_{22}(\lambda).
    \end{split}
    \end{equation}
    Upon recalling the identities $(1 - e^{\pm i\lambda})^{-2d} = 2^{-2d} \sin^{-2d}\left( \frac{\lambda}{2}  \right) e^{\mp i d (\lambda - \pi) }$ (e.g., (9.4.25) in \cite{pipiras_taqqu_2017}) and $(1 -  \cos(\lambda))^{-d} = 2^{-d} \sin^{-2d} \left( \frac{\lambda}{2}  \right)$, we obtain 
    \begin{equation}
    \begin{split}
            f_{11}(\lambda)  &= \frac{1}{\pi} \left[ \left( q_0^2 - q_1^2 \right) \left( 2^{-2d} \sin^{-2d}\left( \frac{\lambda}{2} \right) \cos\left((\lambda - \pi) d\right)   \right) + 2^{-2d} \left( q_0^2 + q_1^2 \right)  \sin^{-2d} \left( \frac{\lambda}{2}  \right) \right] \\
            &= \frac{2^{-2d} \sin^{-2d} \left( \frac{\lambda}{2}  \right)}{\pi}  \left[ (q_0^2 - q_1^2) \cos\left((\lambda - \pi) d\right) + (q_0^2+q_1^2)   \right].
    \end{split}
    \end{equation}
    Moreover, since in addition $g_2(\lambda) = ( -(1-e^{-i\lambda})^{-d}q_1 +(1-e^{i\lambda})^{-d}q_1, (1-e^{-i\lambda})^{-d}q_0 +(1-e^{i\lambda})^{-d}q_0)$,
     \begin{equation}
     \begin{split}
            f_{12}(\lambda) &= \frac{1}{2\pi} g_1(\lambda) g_2^*(\lambda) \\
            &= \frac{1}{2\pi} \left[ 2 q_0 q_1 \left( (1-e^{-i\lambda})^{-2d} - (1-e^{i\lambda})^{-2d} \right)   \right] \\
            &= \frac{i}{\pi} \left[  q_0 q_1 2^{1-2d} \sin^{-2d}\left( \frac{\lambda}{2} \right) \sin\left((\lambda - \pi) d \right)  \right] \\
            &= - f_{21}(\lambda).
    \end{split}
    \end{equation}
    The relation \eqref{eq:g0-g1-limiting-model} is immediate upon recalling $r_0,r_1$ from \eqref{eq:def-a-b} and the formulas \eqref{eq:g0-g1-r0-r1} for $g_0,g_1$ of Proposition \ref{prop:r0-r1-a0-a1-g0-g1}.
\end{proof}

\subsubsection{``Boundary" case} \label{subsubsec:boundary}

Note that in the ``boundary" case $\phi = \pm \left( \frac{1}{2} - d \right) \pi$, since $r_0 = \mp \tan(\pi d) r_1$ by \eqref{eq:borderline-phi-g-r} and in view of \eqref{eq:def-a-b}, we obtain 
\begin{equation}
    \left( q_0 \sqrt{\frac{1}{\cos(\pi d)} + 1} \mp q_1 \sqrt{\frac{1}{\cos(\pi d)} - 1}    \right)^2 =0,
\end{equation}
corresponding to
\begin{equation}\label{eq:q0-boundary}
    q_0 = \pm \sqrt{\frac{1 - \cos(\pi d)}{1 + \cos(\pi d)}} q_1 = \pm \frac{\sin(\pi d)}{1 + \cos(\pi d)} q_1.
\end{equation}
Therefore, in the ``boundary" case,
 \begin{equation} \label{eq:q_0^2+q_1^2}
     q_0^2 - q_1^2 = q_1^2 \left( \frac{1 - \cos(\pi d)}{1 + \cos(\pi d)} - 1  \right) = - \frac{2 \cos(\pi d)}{1+ \cos(\pi d)} q_1^2, \quad q_0^2 + q_1^2 = q_1^2 \frac{2}{1 + \cos(\pi d)}.
 \end{equation}
Due to \eqref{eq:acvf-parametric-finite-lag} and \eqref{eq:q_0^2+q_1^2}, we have, for $h \in \ZZ$,
 \begin{equation} \label{eq:acvf-parametric-finite-lag-borderline}
 \begin{split}
     \gamma_{Y,11}(h) &=  q_1^2 \left( \frac{4 \Gamma(1-2d) \sin(\pi d)}{1 + \cos(\pi d)}  \frac{\Gamma(h+d) }{ \Gamma(h+1-d) \pi} - \frac{2 \cos(\pi d)}{1+ \cos(\pi d)}\frac{\Gamma(2d+h)}{\Gamma(2d) \Gamma(1+h)}  \1_{\{h \ge 0\}} \right. \\
     &\hspace{4cm}\left. - \frac{2 \cos(\pi d)}{1+ \cos(\pi d)}  \frac{\Gamma(2d-h)}{\Gamma(2d) \Gamma(1-h)} \1_{\{h \le 0\}} \right),  \\
    \gamma_{Y,12}(h) &=  
       \pm \text{sign}(h) \sqrt{\frac{1 - \cos(\pi d)}{1 + \cos(\pi d)}}  \frac{ 2 q^2_1}{ \Gamma(2d)} \frac{\Gamma(h+2d)}{\Gamma(1+h)}.
  \end{split}
 \end{equation}
The relations in \eqref{eq:q_0^2+q_1^2} lead to the following spectral densities in \eqref{eq:param-spectral} for $\lambda > 0$,
 \begin{equation} \label{eq:spec-Y-boundary}
 \begin{split}
 f_{Y,11}(\lambda) &= \frac{2^{1-2d} \sin^{-2d}\left( \frac{\lambda}{2} \right) q_1^2}{\pi} \frac{1 - \cos(\pi d) \cos\left((\lambda - \pi) d \right)}{1 + \cos(\pi d)},
  \\
 f_{Y,12}(\lambda) &= \pm i \frac{2^{1-2d} \sin^{-2d}\left( \frac{\lambda}{2}\right) q_1^2}{\pi} \frac{\sin(\pi d) \sin\left((\lambda - \pi) d \right)}{1 + \cos(\pi d)}  .
  \end{split}
 \end{equation}
 
We now investigate further the spectral density of the parametric RMod series in the ``boundary" case. This will be useful in Section \ref{subsec:asymm-mem-par} below. First, let $r_0 =\tan(\pi d) r_1$ (i.e., $\phi = -\left( \frac{1}{2} - d \right) \pi$), corresponding to  the case 
\begin{equation} \label{eq:fy-12-bound}
    f_{Y,12}(\lambda) =  -i \frac{2^{1-2d} \sin^{-2d}\left( \frac{\lambda}{2}\right) q_1^2}{\pi} \frac{\sin(\pi d) \sin\left((\lambda - \pi) d \right)}{1 + \cos(\pi d)}. 
\end{equation}
By combining \eqref{eq:f-X-from-Y-general} with \eqref{eq:spec-Y-boundary}, we have, as $\lambda \to \lambda_0^+$,
\begin{equation} \label{eq:666}
\begin{split}
    &\quad\frac{1}{2} f_{Y,11} (\lambda - \lambda_0) - \frac{1}{2i} f_{Y,12} (\lambda - \lambda_0) \\
    &= \frac{2^{-2d} \sin^{-2d}\left( \frac{\lambda - \lambda_0 }{2} \right) q_1^2}{\pi} \frac{1 - \cos(\pi d) \cos\left((\lambda - \lambda_0 - \pi) d \right) + \sin(\pi d) \sin\left((\lambda - \lambda_0 - \pi) d \right)}{1 + \cos(\pi d)} \\
    &= \frac{2^{-2d} \sin^{-2d}\left( \frac{\lambda - \lambda_0}{2} \right) q_1^2}{\pi ( 1 + \cos(\pi d) )} \left( 1 - \cos\left( (\lambda - \lambda_0 ) d \right) \right) \\
    &\sim  \frac{d^2}{2\pi} \frac{q_1^2}{(1 + \cos(\pi d))} \left( \lambda - \lambda_0 \right)^{-2d+2} = o(1), 
\end{split}
\end{equation}
where for the last line we have used the asymptotic relations
\[
\sin^{-2d}\left(\frac{\lambda - \lambda_0}{2}\right) \sim \left(\frac{\lambda - \lambda_0}{2}\right)^{-2d}, \quad 1 - \cos((\lambda - \lambda_0) d ) \sim \frac{d^2}{2}(\lambda - \lambda_0)^2 , \quad \text{as } \lambda \to \lambda_0^+.
\]
In addition, as $\lambda \to \lambda_0^+$,
\begin{equation}
\begin{split}
     \frac{1}{2} f_{Y,11} (\lambda + \lambda_0) + \frac{1}{2i} f_{Y,12} (\lambda + \lambda_0) &= \frac{2^{-2d} \sin^{-2d}\left( \frac{\lambda + \lambda_0}{2} \right) q_1^2}{\pi ( 1 + \cos(\pi d) )} \left( 1 - \cos\left( (2\pi  -(\lambda + \lambda_0 ) ) d \right) \right) \\
    &\sim  \frac{2^{-2d} \sin^{-2d}\left( \lambda_0 \right) }{\pi ( 1 + \cos(\pi d) )} \left( 1 - \cos\left( 2(\pi -\lambda_0 )d \right) \right) q_1^2 = O(1).
\end{split}
\end{equation}
This shows that, from \eqref{eq:f-X-from-Y-general},
\begin{equation} \label{eq:667}
\begin{split}
   f_X(\lambda) &=  \left[ \frac{1}{2} f_{Y,11} (\lambda - \lambda_0) - \frac{1}{2i} f_{Y,12} (\lambda - \lambda_0) + \frac{1}{2} f_{Y,11} (\lambda + \lambda_0) + \frac{1}{2i} f_{Y,12} (\lambda + \lambda_0)
  \right] \\
  &\sim \frac{2^{-2d} \sin^{-2d}\left( \lambda_0 \right) }{\pi ( 1 + \cos(\pi d) )} \left( 1 - \cos\left( 2(\pi -\lambda_0 )d \right) \right) q_1^2 = O(1), \quad \text{as }\lambda \to \lambda_0^+. 
\end{split}
\end{equation}
Thus, $c_f^+ = 0$ in \eqref{eq:cyclical-spectrum}. When $r_0 =\tan(\pi d) r_1$ and $\lambda < 0$,
\begin{equation}
f_{Y,12}(\lambda) = f_{Y,21}(-\lambda) = - f_{Y,12}(-\lambda),
\end{equation}
which implies from \eqref{eq:fy-12-bound} that, for $\lambda < 0$,
\[
f_{Y,12}(\lambda) =  i \frac{2^{1-2d} \sin^{-2d}\left( \frac{-\lambda}{2}\right) q_1^2}{\pi} \frac{\sin(\pi d) \sin\left((-\lambda - \pi) d \right)}{1 + \cos(\pi d)},
\]
and similarly $f_{Y,11}(\lambda) = f_{Y,11}(-\lambda)$. Then, as $\lambda \to \lambda_0^-$,
\begin{equation} \label{eq:668}
\begin{split}
    f_X(\lambda) &\sim \frac{2^{-2d} \sin^{-2d}\left( \frac{\lambda_0 - \lambda }{2} \right) q_1^2}{\pi} \frac{1 - \cos(\pi d) \cos\left((\lambda_0 - \lambda - \pi) d \right) - \sin(\pi d) \sin\left((\lambda_0 - \lambda - \pi) d \right)}{1 + \cos(\pi d)} \\
    &= \frac{2^{-2d} \sin^{-2d}\left( \frac{\lambda_0 - \lambda}{2} \right) q_1^2}{\pi ( 1 + \cos(\pi d) )} \left( 1 - \cos\left( 2\pi d - (\lambda_0 - \lambda ) d \right) \right) \\
    &\sim  \frac{  q_1^2 (1 - \cos(2 \pi d))}{\pi ( 1 + \cos(\pi d) )} \left( \lambda_0 - \lambda \right)^{-2d} \doteq c_f^- \left( \lambda_0 - \lambda \right)^{-2d}.
\end{split}
\end{equation}
Thus, \eqref{eq:cyclical-spectrum} also holds with $c_f^- > 0$ (assuming $q_1 \neq 0$).

Finally, we note that in the case $r_0 = \tan(\pi d) r_1$ (i.e., $\phi = \left(1-d\right)\pi$), we can reverse the calculations as $\lambda \to \lambda_0^+, \; \lambda \to \lambda_0^-$, yielding
\begin{equation} \label{eq:bound-case-spec-dens}
\begin{split}
    f_X(\lambda) &\sim  \frac{ q_1^2  (1 - \cos(2 \pi d))}{\pi ( 1 + \cos(\pi d) )} \left( \lambda - \lambda_0 \right)^{-2d} \doteq c_f^+ \left( \lambda - \lambda_0 \right)^{-2d}, \quad \text{as } \lambda \to \lambda_0^+, \\
    f_X(\lambda) &\sim  \frac{2^{-2d} \sin^{-2d}\left( \lambda_0 \right) }{\pi ( 1 + \cos(\pi d) )} \left( 1 - \cos\left( 2(\pi - \lambda_0) d \right) \right) q_1^2, \quad \text{as } \lambda \to \lambda_0^- .
\end{split}
\end{equation}

\subsubsection{Adding short memory effects}
\label{subsubsec:short-range-effects}

The parametric model for the bivariate LM series $Y$ in \eqref{eq:Yn}--\eqref{eq:def-Q+-} and the resulting RMod series $X$ in \eqref{eq:cycl-constr-basic} were constructed in Section \ref{subsubsec:clm-effects} to capture the general form of CLM, which is the asymptotic notion in the time and spectral domains having to do with long memory. Short memory effects can be added to the model in a standard way. For this purpose, it is convenient to introduce the following terminology and models akin to fractionally integrated ARMA (AutoRegressive Moving Average) models. Let
\begin{align}
    \Phi(z) = 1 - \phi_1 z - \dots - \phi_p z^p, \\
    \Theta(z) = 1 + \theta_1 z + \dots + \theta_q z^q,
\end{align}
be the AR and MA polynomials of orders $p,q \in \NN_0$. Assume their roots are outside the unit disc.

\begin{defn}\label{d:model}
    The RMod series in \eqref{eq:cycl-constr-basic} constructed from the bivariate LM series in \eqref{eq:Yn}--\eqref{eq:def-Q+-} will be called fractional RMod series and denoted as FRMod$(0,d,0)$. The second-order stationary series $X$ will be denoted FRMod$(p,d,q)$ and called (cyclical) fractional RMod series of orders $p,d,q$ if the series
    \begin{equation}
        \Theta^{-1}(B) \Phi(B) X_n
    \end{equation}
    is FRMod$(0,d,0)$.
\end{defn}

When $p = 0$, the FRMod$(0,d,q)$ series can be written as
\begin{equation} \label{eq:frmod-0dq}
    X_n = \Theta(B) \tilde X_n = \tilde X_n + \Theta_1 \tilde X_{n-1} + \dots + \Theta_q \tilde X_{n-q},
\end{equation}
where $\tilde X$ is FRMod$(0,d,0)$. Since the ACVF of $\tilde X$ can be represented explicitly, the same holds for the series \eqref{eq:frmod-0dq}. The polynomial $\Theta(B)$ allows modeling the ACVF for smaller lags and the spectral density away from the divergence around the cyclical phase $\lambda_0$. When $q = 0$, the FRMod$(p,d,0)$ model can be fitted in practice since 
\begin{equation}
    \Phi(B) X_n = X_n - \phi_1 X_{n-1} - \dots - \phi_p X_{n-p}
\end{equation}
is FRMod$(0,d,0)$ with a known ACVF.


\subsection{Insights From Random Modulation}
\label{subsec:RM-insights}

In Sections \ref{subsec:cycl-constr} and \ref{subsec:bivar-constr}, we constructed CLM models by random modulation of a bivariate LM series satisfying certain ACVF properties. This raises the question: Is the converse procedure also possible, i.e., given a CLM series with cyclical frequency $\lambda_0$, can it be represented as a random modulation of a LM bivariate series $(Y_{1,n},Y_{2,n})$? The answer to this question will follow from general developments for RMod series as, for example, in \cite{Pap83}.

Thus, let $\{X_{1,n}\}$ be a CLM series. Assume $\{X_{2,n}\}$ is another series such that the bivariate series $\{(X_{1,n},X_{2,n})\}$ satisfies property \eqref{Property1}. Examples of the series $\{X_{2,n}\}$ are given below. Set
\begin{equation} \label{eq:converse-construction}
\begin{split}
    Y_{1,n} &= \cos(\lambda_0 n) X_{1,n} - \sin(\lambda_0 n) X_{2,n}, \\
    Y_{2,n} &= \sin(\lambda_0 n) X_{1,n} + \cos(\lambda_0 n) X_{2,n},
\end{split}
\end{equation}
which is equivalent to
\begin{equation} \label{eq:converse-construction-X-from-Y}
\begin{split}
    X_{1,n} &= \cos(\lambda_0 n) Y_{1,n} + \sin(\lambda_0 n) Y_{2,n}, \\
    X_{2,n} &= - \sin(\lambda_0 n) Y_{1,n} + \cos(\lambda_0 n) Y_{2,n}.
\end{split}
\end{equation}
If $\{(X_{1,n},X_{2,n})\}$ satisfies property \eqref{Property1}, one can check that the bivariate series $\{Y_n\} = \{(Y_{1,n},Y_{2,n})\}$ is second-order stationary (e.g., \cite{Pap83}). The proposition below will show that the bivariate series $\{Y_n\}$ is LM when $\{X_{1,n}\}$, $\{X_{2,n}\}$ are CLM, that is, in view of the first equation in \eqref{eq:converse-construction}, answering the question above in the affirmative.

Different choices of the series $\{X_{2,n}\}$ above are available in the random modulation literature and lead to different bivariate series. Popular choices are:
\begin{itemize}
    \item Uncorrelated case: Take $\{X_{2,n}\}$ uncorrelated with $\{X_{1,n}\}$. In this casee, the ACVF of $\{(X_{1,n},X_{2,n})\}$ is: for all $h \in \ZZ$,
    \begin{equation} \label{eq:uncorr-insights}
        \gamma_{X,22}(h) = \gamma_{X,11}(h), \quad \gamma_{X,12}(h) = 0.
    \end{equation}
    Similarly, the spectral density satisfies $f_{X,22}(\lambda) = f_{X,11}(\lambda)$ and $f_{X,12}(\lambda) = 0$.

    \item Hilbert transform case: Take $\{X_{2,n}\}$ to be the Hilbert transform of $\{X_{1,n}\}$. In this case, the spectral density of $\{(X_{1,n},X_{2,n})\}$ satisfies
    \begin{equation} \label{eq:hilb-transf-spec}
        f_{X,22}(\lambda) = f_{X,11}(\lambda),\;  f_{X,12}(\lambda) = - f_{X,21} (\lambda) = i f_{X,11}(\lambda) \text{sign}(\lambda).
    \end{equation}
    We also note that in the random modulation literature, $\lambda_0$ in \eqref{eq:converse-construction} and \eqref{eq:converse-construction-X-from-Y} is a free parameter. Here, we fix it to be equal to the cyclical frequency. The representation of $\{X_{1,n}\}_{n \in \ZZ}$ in \eqref{eq:converse-construction-X-from-Y} with $\{X_{2,n}\}_{n \in \ZZ}$ as the Hilbert transform of $\{X_{1,n}\}_{n \in \ZZ}$ is referred to as \textit{Rice's representation}. 
\end{itemize}


\begin{proposition}
Let $\{X_{1,n}\}$ be a CLM series satisfying \eqref{eq:cyclical-spectrum} and let $\{X_{2,n}\}$ be the series constructed in the uncorrelated or Hilbert transform case above. Let also $\{Y_n \}_{n \in \ZZ} = \{(Y_{1,n},Y_{2,n})\}_{n \in \ZZ}$ be given by \eqref{eq:converse-construction}. Then, $\{Y_n\}_{n \in \ZZ}$ is a bivariate LM series satisfying \eqref{eq:biv-spectral-general} with \eqref{eq:spectral-asympotitc-Y}. Furthermore, we have:
\begin{itemize}
    \item Uncorrelated case: 
    \begin{equation} \label{eq:g0g1-insights-uncorr}
        g_0 = \frac{1}{2} \left( c_f^+ + c_f^- \right), \; g_1 = \frac{1}{2} \left( c_f^- - c_f^+   \right).
    \end{equation}
    \item Hilbert transform case:
    \begin{equation} \label{eq:g0g1-insight-hilbert}
          g_0 =  c_f^+ + c_f^-  , \; g_1 = c_f^- - c_f^+ .
    \end{equation}
\end{itemize}
\end{proposition}

\begin{proof}
    We start by showing that $\{Y_n \}_{n \in \ZZ}$ is a bivariate LM series satisfying \eqref{eq:biv-spectral-general} with \eqref{eq:spectral-asympotitc-Y}. Indeed, we see that
\begin{equation} \label{eq:conv-LM}
\begin{split}
f_{Y}(\lambda) &= \sum_{h \in \ZZ} e^{-ih\lambda} \gamma_{Y} (h) \\
&= \sum_{h \in \ZZ} e^{-i \lambda h}
\begin{pmatrix}
    \gamma_{X,11}(h) \cos(\lambda_0h) - \gamma_{X,12}(h) \sin(\lambda_0 h) & \gamma_{X,12}(h) \cos(\lambda_0h) - \gamma_{X,11}(h) \sin(\lambda_0 h) \\
    -\gamma_{X,12}(h) \cos(\lambda_0h) +\gamma_{X,11}(h) \sin(\lambda_0 h) & \gamma_{X,11}(h) \cos(\lambda_0h) - \gamma_{X,12}(h) \sin(\lambda_0 h) 
\end{pmatrix} \\
&= \begin{pmatrix}
    f_{Y,11}(\lambda) & f_{Y,12}(\lambda) \\
    f_{Y,21}(\lambda) & f_{Y,22}(\lambda)
\end{pmatrix},
\end{split}
\end{equation}
where the second line follows from standard calculations as in the proof of Proposition \ref{thm:constr}. Moreover,
    \begin{equation} \label{eq:conv-LM-A}
\begin{split}
    f_{Y,11} (\lambda) &=  \sum_{h \in \ZZ} \left(  \frac{1}{2} \left[\gamma_{X,11} (h)  \left( e^{- i(\lambda - \lambda_0) h} + e^{-i (\lambda + \lambda_0)  h} \right) \right]- \frac{1}{2i} \left[   \gamma_{X,12} (h) \left( e^{- i(\lambda - \lambda_0) h} - e^{-i (\lambda + \lambda_0)  h}\right) \right] \right) \\
    &= \frac{1}{2} \left[ f_{X,11}(\lambda - \lambda_0) +  f_{X, 11}(\lambda + \lambda_0) \right] - \frac{1}{2i} \left[ f_{X,12}(\lambda - \lambda_0) -  f_{X, 12}(\lambda + \lambda_0) \right]
\end{split}
\end{equation}
and 
\begin{equation} \label{eq:conv-LM-B}
\begin{split}
    f_{Y,12}(\lambda) &=  \sum_{h \in \ZZ} \left( \frac{1}{2} \left[ \gamma_{X,12} (h)  \left( e^{- i(\lambda - \lambda_0) h} + e^{-i (\lambda + \lambda_0)  h} \right) \right] - \frac{1}{2i} \left[  \gamma_{X,11} (h) \left( e^{- i(\lambda - \lambda_0) h} - e^{-i (\lambda + \lambda_0)  h}\right) \right] \right) \\
    &= \frac{1}{2} \left[ f_{X,12}(\lambda - \lambda_0) + f_{X,12}(\lambda + \lambda_0) \right] - \frac{1}{2i} \left[ f_{X,11}(\lambda - \lambda_0) - f_{X, 11}(\lambda + \lambda_0) \right].
\end{split}
\end{equation} 
From \eqref{eq:conv-LM}, \eqref{eq:conv-LM-A}, and \eqref{eq:conv-LM-B}, we see that, in the uncorrelated case \eqref{eq:uncorr-insights}, as $\lambda \to 0^+$,
\begin{equation}
    f_{Y,11}(\lambda) \sim  
       \frac{1}{2}\left( c_f^+ + c_f^- \right) |\lambda|^{-2d},
\quad 
   f_{Y,12}(\lambda) \sim  
       \frac{i}{2} \left( c_f^- - c_f^+   \right) |\lambda|^{-2d},
\end{equation}
which leads to $g_0,g_1$ given in \eqref{eq:g0g1-insights-uncorr}.

On the other hand, considering the Hilbert transform case \eqref{eq:hilb-transf-spec}, we see that, as $\lambda \to 0^+$,
\begin{equation}
    f_{Y,11}(\lambda) \sim  
         \left( c_f^+ + c_f^- \right) |\lambda|^{-2d}, \quad 
   f_{Y,12}(\lambda) \sim  
       i \left( c_f^- - c_f^+  \right) |\lambda|^{-2d}, 
\end{equation}
leading to the respective constants $g_0,g_1$ presented in \eqref{eq:g0g1-insight-hilbert}.
\end{proof} 

\begin{remark}
    From the perspective of signal processing, there are physical reasons to select the Hilbert transform in random modulation; see the discussion in \cite{Pap83}. It is not immediately clear whether there is a statistical benefit in selecting the Hilbert transform in the context considered here.
\end{remark}

\subsection{Linear Representations}

We identify here the linear representation for the parametric RMod series $\{X_n\}_{n \in \ZZ}$ when the series $Y$ is constructed in \eqref{eq:D-def-diag}--\eqref{eq:def-Q+-}. First note that the construction in \eqref{eq:D-def-diag}--\eqref{eq:def-Q+-} is a special case of the general form in \eqref{eq:Zn}--\eqref{eq:A+A-general} with (upon noticing that $c_{d,0} = 1$)
\begin{equation} \label{eq:ak-q0-q1-cdk}
a_{0,k} = 
\begin{cases}
    c_{d,-k} q_0 & k \le -1, \\
    2q_0 & k = 0, \\
    c_{d,k} q_0 & k \ge 1,
\end{cases} \quad 
a_{1,k} = \operatorname{sign}(k) c_{d,|k|} q_1 =
\begin{cases}
    -c_{d,-k} q_1 & k \le -1, \\
    0 & k = 0,\\
    c_{d,k} q_1 & k \ge 1.
\end{cases}
\end{equation}
The following result provides a linear representation for the RMod Series $\{X_n\}_{n \in \ZZ}$ and is an immediate corollary of Theorem \ref{thm:cycl-linear-general}.
\begin{corollary}
    Let $\{X_n\}_{n \in \ZZ}$ and $\{Y_n\}_{n \in \ZZ}$ be given by \eqref{eq:D-def-diag} and \eqref{eq:cdk-asympt}--\eqref{eq:def-Q+-} respectively. Then, $\{X_n\}_{n \in \ZZ}$ has the linear representation obtained in Theorem \ref{thm:cycl-linear-general} with $a_{0,k},a_{1,k}, k \in \ZZ$, given in \eqref{eq:ak-q0-q1-cdk}.
\end{corollary}

We finally note that this general representation retrieves a two-sided linear representation for the fractional sinusoidal waveform process; see, e.g., Proposition 1 of \cite{ProMad22}.


\section{Extensions}
\label{sec:extensions}

\subsection{Asymmetry in Memory Parameter} \label{subsec:asymm-mem-par}

Random modulation and our parametric approach can also be used to construct CLM series satisfying \eqref{eq:cyclical-spectrum-assym}, that is,
\begin{equation} \label{eq:cyclical-spectrum-assym-2}
 f_X(\lambda) \sim 
 \begin{cases}
       c_{f,+} (\lambda - \lambda_0)^{-2d_+}, &\quad\text{as}\; \lambda \to \lambda_0^+,\\
       c_{f,-} (\lambda_0 - \lambda)^{-2d_-}, &\quad\text{as}\; \lambda \to \lambda_0^-,
 \end{cases}
\end{equation}
where $c_{f,+},c_{f,-} > 0$ and $d_-,d_+ \in \left(0,\frac{1}{2}\right)$, with $d_- \neq d_+$. As noted in Section \ref{sec:intro}, it is natural to set
\begin{equation} \label{eq:asymm-memory-model}
    X_n = X_{n}^+ + X_{n}^-,
\end{equation}
where $\{X^+_n\}_{n \in \ZZ}$ and $\{X^-_n\}_{n \in \ZZ}$ are uncorrelated and satisfy \eqref{eq:cyclical-spectrum-assym-constr}. Recall that the asymptotics \eqref{eq:cyclical-spectrum-assym-constr} are associated with the ``boundary" cases \eqref{eq:cycl-phi-boundary}. The asymptotic behavior $0 \cdot |\lambda-\lambda_0 |^{-2d_\pm}$ in \eqref{eq:cyclical-spectrum-assym-constr}, however, is not specific enough to guarantee \eqref{eq:cyclical-spectrum-assym-2}. Conditions implying \eqref{eq:cyclical-spectrum-assym-2} are
\begin{equation} \label{eq:cyclical-spectrum-assym-constr2}
 f_{X^+}(\lambda) \sim 
 \begin{cases}
       c_{f,+} (\lambda - \lambda_0)^{-2d_+}, &\quad\text{as}\; \lambda \to \lambda_0^+,\\
       b_{f,+} \cdot (\lambda_0 - \lambda)^{0}, &\quad\text{as}\; \lambda \to \lambda_0^-,
 \end{cases} \quad
  f_{X^-}(\lambda) \sim 
 \begin{cases}
       b_{f,-} \cdot  (\lambda - \lambda_0)^{0}, &\quad\text{as}\; \lambda \to \lambda_0^+,\\
       c_{f,-} (\lambda_0 - \lambda)^{-2d_-}, &\quad\text{as}\; \lambda \to \lambda_0^-,
 \end{cases}
\end{equation}
where $b_{f,+},b_{f,-} \ge 0$. We argue here that the parametric RMod model constructed in Section \ref{subsubsec:boundary} in the ``boundary" cases \eqref{eq:cycl-phi-boundary} satisfy \eqref{eq:cyclical-spectrum-assym-constr2}, and thus yields the model \eqref{eq:asymm-memory-model} satisfying \eqref{eq:cyclical-spectrum-assym-2}.

Following the development of Section \ref{subsubsec:boundary}, we take $\{X^+_n\}$ in \eqref{eq:asymm-memory-model} as a parametric model FRMod$(0,d,0)$ constructed in Section \ref{subsubsec:clm-effects} with
\begin{equation} \label{eq:d+-q0-q1}
    d = d_+, \quad q_{0,+} =  \frac{\sin(\pi d_+)}{1 + \cos(\pi d_+)}q_{1,+}.
\end{equation}
By \eqref{eq:bound-case-spec-dens}, the spectral density of $\{ X_n^+\}_{n \in \ZZ}$ satisfies \eqref{eq:cyclical-spectrum-assym-constr2} with
\begin{equation} \label{eq:bound-case-spec-dens-asymm}
\begin{split}
    c_{f,+}  = \frac{ q_{1,+}^2  (1 - \cos(2 \pi d_+))}{\pi ( 1 + \cos(\pi d_+) )},  \quad
    b_{f,+} = \frac{2^{-2d_+} \sin^{-2d_+}\left( \lambda_0 \right) }{\pi ( 1 + \cos(\pi d_+) )} \left( 1 - \cos\left( 2(\pi - \lambda_0) d_+ \right) \right) q_{1,+}^2.
\end{split}
\end{equation}
Similarly, we take $\{X_n^-\}$ in \eqref{eq:asymm-memory-model} as the same parametric model with
\begin{equation} \label{eq:d--q0-q1}
    d = d_-, \quad q_{0,-} = - \frac{\sin(\pi d_-)}{1 + \cos(\pi d_-)}q_{1,-}.
\end{equation}
By \eqref{eq:667}--\eqref{eq:668}, the spectral density of $\{X_n^-\}$ satisfies \eqref{eq:cyclical-spectrum-assym-constr2} with
\begin{equation} \label{eq:bound-case-spec-dens-asymm2}
\begin{split}
    c_{f,-}  =\frac{  q_{1,-}^2 (1 - \cos(2 \pi d_-))}{\pi ( 1 + \cos(\pi d_-) )},  \quad
    b_{f,-} = \frac{2^{-2d_-} \sin^{-2d_-}\left( \lambda_0 \right) }{\pi ( 1 + \cos(\pi d_-) )} \left( 1 - \cos\left( 2(\pi -\lambda_0 )d_- \right) \right) q_{1,-}^2.
\end{split}
\end{equation}
We can thus construct the series \eqref{eq:asymm-memory-model} satisfying \eqref{eq:cyclical-spectrum-assym-2} with $c_{f,+}$ and $c_{f,-}$ given in \eqref{eq:bound-case-spec-dens-asymm} and \eqref{eq:bound-case-spec-dens-asymm2}. The construction above could be extended to FRMod$(p,d,q)$ models in the ``boundary" case.

One can also compute the ACVF of $X$. Indeed, let $\{Y^+_n\}_{n \in \ZZ}$ and $\{Y^-_n\}_{n \in \ZZ}$ be given by \eqref{eq:D-def-diag}--\eqref{eq:def-Q+-}, with parameters specified by \eqref{eq:d+-q0-q1} and \eqref{eq:d--q0-q1} respectively. Then, we write
\begin{equation} \label{eq:gamma-x+-x-}
\gamma_X(h) = \gamma_{X^+}(h) + \gamma_{X^-}(h) = \cos(\lambda_0 h) \left( \gamma_{Y^+,11}(h) + \gamma_{Y^-,11}(h)   \right) - \sin(\lambda_0 h) \left( \gamma_{Y^+,12}(h) + \gamma_{Y^-,12}(h)   \right),
\end{equation}
where the second equality follows from the first line of \eqref{eq:rmod-basic} for $\gamma_{X^+}(h)$ and $\gamma_{X^-}(h)$ separately. Now define
\begin{equation}
\begin{split}
    \mathfrak{A}_{\pm} &\doteq \frac{4 \Gamma(1-2d_\pm) \sin(\pi d_\pm)}{1 + \cos(\pi d_\pm)}  \frac{\Gamma(h+d_\pm) }{ \Gamma(h+1-d_\pm) \pi}, \quad  
    \mathfrak{B}_{\pm} \doteq \frac{2 \cos(\pi d_\pm)}{1+ \cos(\pi d_\pm)}\frac{\Gamma(2d_\pm +h)}{\Gamma(2d_\pm) \Gamma(1+h)} \1_{\{h \ge 0\}}, \\
    \mathfrak{C}_{\pm} &\doteq \frac{2 \cos(\pi d_\pm)}{1+ \cos(\pi d_\pm)}\frac{\Gamma(2d_\pm - h)}{\Gamma(2d_\pm) \Gamma(1 - h)}  \1_{\{h \le 0\}} , \quad
    \mathfrak{D}_{\pm} \doteq \sqrt{\frac{1 - \cos(\pi d_\pm)}{1 + \cos(\pi d_\pm)}}  \frac{ 2 }{ \Gamma(2d_\pm)} \frac{\Gamma(h+2d_\pm)}{\Gamma(1+h)}.
\end{split}
\end{equation}
In view of \eqref{eq:acvf-parametric-finite-lag-borderline}, this leads to, for $h \in \ZZ$,
\begin{equation}
         \gamma_X(h) = \cos(\lambda_0 h) \left[ q^2_{1,+} \left( \mathfrak{A}_+ - \mathfrak{B}_+ - \mathfrak{C}_+   \right) + q^2_{1,-} \left( \mathfrak{A}_- - \mathfrak{B}_- - \mathfrak{C}_-   \right)   \right] - \sin(\lambda_0 h)\text{sign}(h) \left[] q_{1,+}^2 \mathfrak{D}_+ -  q_{1,-}^2 \mathfrak{D}_-  \right].
\end{equation}
 To identify the asymptotics of $\gamma_X(h)$, we need more refined approximations than the ones expressed through \eqref{eq:gamma-asymptotics-first order}. From Equation (1) of \cite{Erftri51} and \eqref{eq:gamma-asymptotics-first order}, it follows that, as $h \to +\infty$,
 \begin{equation} \label{eq:gamma+-second-asympt}
\begin{split}
\gamma_{Y^+,11}(h) &=  \frac{1}{2\pi} \left[ 4 (q_{0,+}^2+q_{1,+}^2) \Gamma(1-2d_+) \sin (d_+\pi)  + 2 (q_{0,+}^2-q_{1,+}^2)  \frac{\pi}{\Gamma(2d_+)}   \right]h^{2d_+-1} + O(h^{2d_+-2}), \\
\gamma_{Y^+,12}(h) &=  \frac{2q_{0,+} q_{1,+}}{\Gamma(2d_+)} h^{2d_+-1} + O(h^{2d_+-2}),
\end{split}
\end{equation}
and similarly
 \begin{equation} \label{eq:gamma--second-asympt}
\begin{split}
\gamma_{Y^-,11}(h) &=  \frac{1}{2\pi} \left[ 4 (q_{0,-}^2+q_{1,-}^2) \Gamma(1-2d_-) \sin (d_-\pi)  + 2 (q_{0,-}^2-q_{1,-}^2)  \frac{\pi}{\Gamma(2d_-)}   \right]h^{2d_--1} + O(h^{2d_--2}), \\
\gamma_{Y^-,12}(h) &= \frac{2q_{0,-} q_{1,-}}{\Gamma(2d_-)} h^{2d_--1} + O(h^{2d_--2}).
\end{split}
\end{equation}
Let $d^* \doteq \max \{d_+,d_-\}$, $d_* \doteq  \min\{d_+,d_-\}$ and $\bar q_{i} = q_{i,\pm}$, $q_{i} = q_{i,\mp}$ if $d^* = d_\pm$. Then, since $2 d^* -2 < 2d_* - 1 < 2 d^* -1$ and from \eqref{eq:gamma-x+-x-}, \eqref{eq:gamma+-second-asympt}, and \eqref{eq:gamma--second-asympt}, we have
\begin{equation}
\begin{split}
  \gamma_X(h) &= h^{2d^* -1} \left[ \cos(\lambda_0 h) \frac{1}{2\pi} \left[ 4 (\bar q_0^2+\bar q_1^2) \Gamma(1-2d^*) \sin (d^*\pi)  + 2 (\bar q_0^2-\bar q_1^2)  \frac{\pi}{\Gamma(2d^*)}   \right] - \sin(\lambda_0 h)  \frac{2\bar q_0 \bar q_1}{\Gamma(2d^*)} \right] \\
  &\quad + h^{2d_* -1} \left[ \cos(\lambda_0 h) \frac{1}{2\pi} \left[ 4 (q_0^2+q_1^2) \Gamma(1-2d_*) \sin (d_*\pi)  + 2 (q_0^2-q_1^2)  \frac{\pi}{\Gamma(2d_*)}   \right] - \sin(\lambda_0 h)  \frac{2q_0 q_1}{\Gamma(2d_*)} \right] \\
  &\quad + \cos(\lambda_0 h) O(h^{2d^* -2}) +\sin(\lambda_0 h) O(h^{2d^* -2}).
\end{split}
\end{equation} \vspace{-.7cm}

The only CLM model we are aware of having the form \eqref{eq:cyclical-spectrum-assym-2} is the Seasonal Cyclical Asymmetric Long Memory (CSALM) model of \cite{Arteche20003}. The CSALM model, unlike the one considered here, does not have explicit ACVF.

\subsection{CLM with Multiple Singularities} \label{subsec:multip-sing}

Our approach extends easily to the case of CLM with multiple singularities, expressed in the spectral domain as: for $\lambda_{0,m} \in (0,\pi)$, $m=1,\dots,M$,
\begin{equation} \label{eq:cyclical-spectrum-multiple}
 f_X(\lambda - \lambda_{0,m}) \sim 
 \begin{cases}
       c_{f,m}^+ \lambda^{-2d_m}, &\quad\text{as}\; \lambda \to 0^+,\\
       c_{f,m}^- ( - \lambda)^{-2d_m}, &\quad\text{as}\; \lambda \to 0^-,
 \end{cases}
\end{equation}
where $d_m \in \left(0,\frac{1}{2} \right)$ and $c_{f,m}^+,c_{f,m}^- \ge 0$, $c_{f,m}^+ + c_{f,m}^- > 0$. We refer to this as CLM with $M$ factors ($M$-CLM, for short). As with \eqref{eq:cycl-ACVF} and \eqref{eq:cyclical-spectrum}, it is expected that under suitable conditions, the condition \eqref{eq:cyclical-spectrum-multiple} is equivalent to the time-domain condition: as $h \to \infty$,
\begin{equation}
    \gamma_X(h) = \sum_{m=1}^M \left( c_{\gamma,m} \cos(\lambda_{0,m} h + \phi_m ) h^{2d_m -1} +  o(h^{2d_m-1}) \right),
\end{equation}
where $\phi_m \in \cli_{d_m}$ and $c_{\gamma,m} > 0$. M-CLM was considered in, e.g., \cite{GirLei95,Arteche20003,Maddanu:2022,ProMad22}.

The series $\{X_n\}_{n \in \ZZ}$ satisfying \eqref{eq:cyclical-spectrum-multiple} can be constructed as
\begin{equation}
    X_n = \sum_{m=1}^M X_{m,n},
\end{equation}
where $\{X_{m,n}\}$, $m=1,\dots,M$, are independent series and the spectral density $f_{X_m}(\lambda)$ of each $\{X_{m,n}\}$ satisfies \eqref{eq:cyclical-spectrum-multiple}. One or more of the parametric RMod models introduced in this work can be taken for $\{X_{m,n}\}$, including in the ``boundary" case. For the resulting parametric model, when $\{X_{m,n}\}$ are constructed in \eqref{eq:Yn}--\eqref{eq:def-Q+-}, the ACVF and spectral densities can be computed explicitly  by
\begin{equation}
    \gamma_X(h) = \sum_{m=1}^M \gamma_{X_m}(h), \quad f_X(\lambda) = \sum_{m=1}^M f_{X_m}(\lambda),
\end{equation}
where $\gamma_{X_m}$ and $f_{X_m}$ are given explicitly in Propositions \ref{thm:Yn-constr} and \ref{prop:spec-dens-param}, respectively. It is evident that one can consider one or more of the $\{X_{m,n}\}$ to be FRMod$(p,d,q)$. Again, explicit formulas for the ACVF of the $M-$factor Gegenbauer processes are not available  (see, e.g., the discussion at the end of Section 3 in \cite{arteche1999} or \cite{Dis18}).

\section{Numerical Illustrations}\label{sec:numerical}

We provide here some simple numerical illustrations for various theoretical developments considered above. 
We do not aim for a detailed simulation study (that would require a rigorous inference treatment that is outside the paper's scope), but rather to shed light on the inner workings of our model. To promote clarity and ease the replication of our work, we will make detailed references to all formulas and share our MATLAB code.

The top plots in Figures \ref{f:fig1}--\ref{f:fig4} depict realizations of FRMod$(0,d,0)$ series $\{X_n\}$, i.e., random modulations \eqref{eq:cycl-constr-basic} with an underlying bivariate LM series 
$\{Y_n\}$ following \eqref{eq:Yn}--\eqref{eq:def-Q+-} for four different parameter schemes. 
The bottom left plots show the respective sample ACVFs and true ACVFs calculated from relations \eqref{eq:rmod-basic} and \eqref{eq:acvf-parametric-finite-lag}. Finally, the bottom right plots show the series' 
periodograms and spectral densities (in log scale) calculated from relations \eqref{eq:f-X-from-Y-general} and \eqref{eq:param-spectral}. The title of each figure contains the values for the FRMod$(0,d,0)$ model parameters $(d,q_0,q_1,\lambda_0)$ as well as the values of the limiting ACVF parameters $c_\gamma$, $\phi$ (calculated from relations \eqref{eq:phi-def} and \eqref{eq:def-a-b}), the set of admissible phases ${\cal I}_d$ (from \eqref{eq:admiss-sets-def}) and the spectral density limiting parameters $c_f^+,c_f^-$ (calculated from \eqref{eq:cfpm-cgamma}).

  \begin{figure}[!h]
      \includegraphics[width=6in, height = 2.5in]{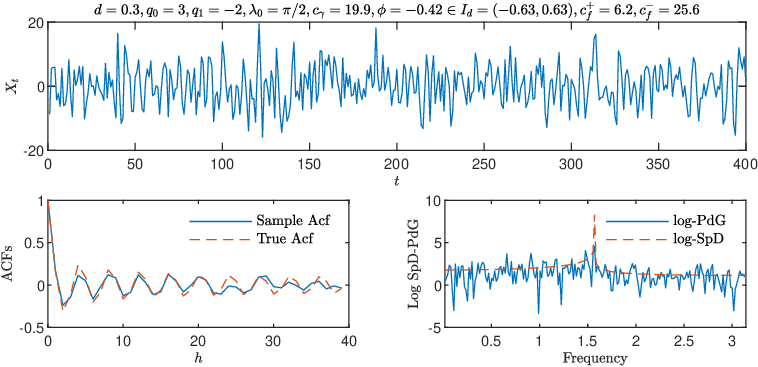}
      \vspace{-2mm}
      \caption{\textit{CLM realization with theoretical ACVF, sample ACVF, spectral density and periodogram.}}\label{f:fig1}  
  \end{figure}

  \begin{figure}[!h]
    \vspace{-1mm}
      \includegraphics[width=6in, height = 2.5in]{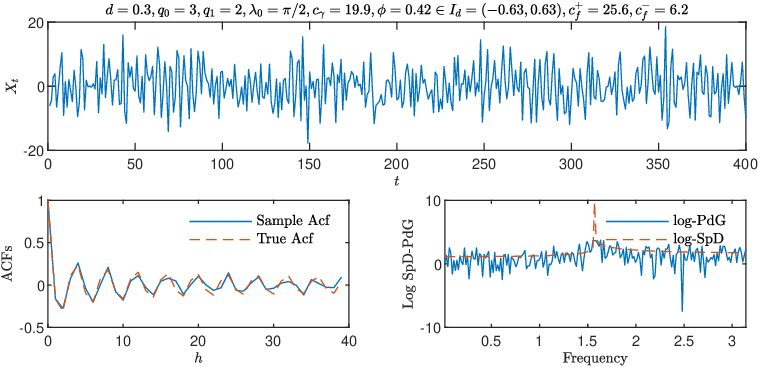}
            \vspace{-2mm}
      \caption{\textit{CLM realization with theoretical ACVF, sample ACVF, spectral density and periodogram.}}\label{f:fig2}  
  \end{figure}  
  
  \begin{figure}[!h]
    \vspace{-1mm}
    \begin{center}
        \includegraphics[width=6in, height = 2.5in]{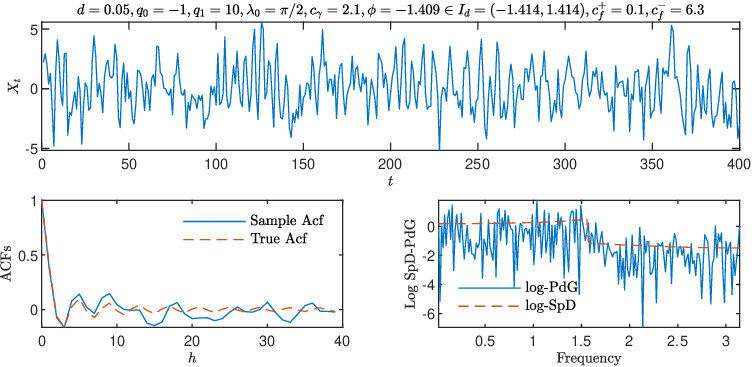}
          \vspace{-2mm}
      \caption{\textit{CLM realization with theoretical ACVF, sample ACVF, spectral density and periodogram.}}\label{f:fig3}  
  \end{center}
  \end{figure}

  \begin{figure}[!h]
    \vspace{-1mm}
    \begin{center}
    \includegraphics[width=6in, height = 2.5in]{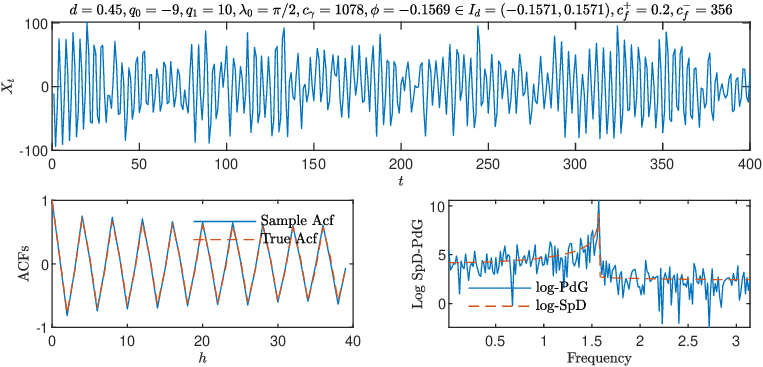}
          \vspace{-2mm}
      \caption{\textit{CLM realization with theoretical ACVF, sample ACVF, spectral density and periodogram.}}\label{f:fig4}  
  \end{center}
\end{figure}

Figures \ref{f:fig5} and \ref{f:fig6} are similar to Figures \ref{f:fig1}--\ref{f:fig4} but correspond to the ``boundary'' case, i.e., $\phi=-(1/2-d)\pi$ and $\phi=(1/2-d)\pi$ respectively, with $d=0.4, q_1=3, \lambda_0 = \pi/4$ and with $q_0, c_f^+,c_f^-$ computed from the formulas \eqref{eq:q0-boundary}, \eqref{eq:668} and \eqref{eq:bound-case-spec-dens}.
Several observations regarding the figures follow.



    The opposite signs of the parameter $q_1$ in Figures
\ref{f:fig1}--\ref{f:fig2} lead to opposite signs in the cyclical phases $\phi$, which in turn control the direction of the asymmetry before and after the cyclical frequency $\lambda_0$. For example, negative values of $\phi$ $(=-0.42)$ lead to spectra with lower frequencies having larger ``weight" than higher ones: witnessed both visually and algebraically through $c_f^- = 25.6 > 6.2 = c_f^+$. This relationship can also be seen in the left plot of Figure \ref{f:fig7} that depicts $q_1$ as a function of $\phi$ for four different $d$'s and is obtained by relations \eqref{eq:def-a-b} and \eqref{eq:phi-def}. 

The dashed green horizontal line in the left plot of Figure \ref{f:fig7}, which is the lower bound of admissible phases for $d=0.35$, intersects the corresponding curve at the minimum value. To avoid overburdening the plot, we did not add similar lines for the upper bound of admissible phases, nor for other values of $d$. However, upon visual inspection, one can see that  restricting $q_1$ at values that yield admissible phases $\phi$ ensures the curves will be one-to-one and thus the model will be identifiable.

The absolute value of $\phi$ does not necessarily control the size of the 
asymmetry gap around the cyclical frequency (in the sense of the average  
difference between a few spectrum values before and after the cyclical frequency). 
Compare, for example, Figures  \ref{f:fig3} and \ref{f:fig4}.

In Figure \ref{f:fig4}, the behavior of the spectrum around the cyclical frequency appears similar
to the boundary case (compare it, for example, with the one in Figure \ref{f:fig5}). This happens because
for the selected values of $q_0$ and $q_1$, relations \eqref{eq:phi-def} and 
\eqref{eq:def-a-b} yield a phase that is close (but not equal) to the lower boundary of the admissible set $\cli_d$ in \eqref{eq:admiss-sets-def}. 


  \begin{figure}[!h]
    \vspace{-1mm}
    \begin{center}
    \includegraphics[width=6in, height = 2.5in]{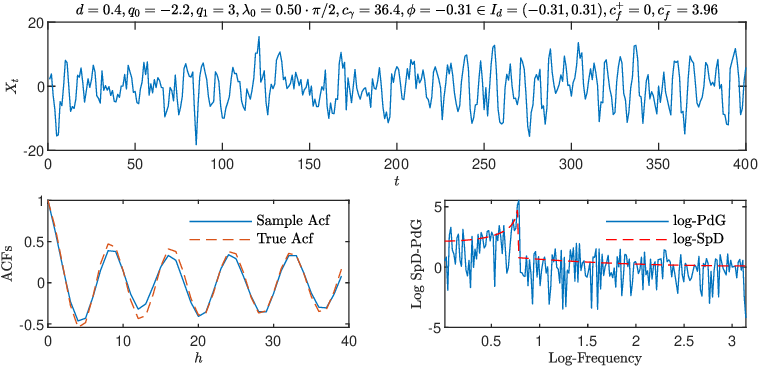}
          \vspace{-2mm}
      \caption{\textit{CLM realization with theoretical ACVF, sample ACVF, spectral density and periodogram.}}\label{f:fig5}  
  \end{center}
  \end{figure}

    \begin{figure}[!h]
    \vspace{-1mm}
    \begin{center}
    \includegraphics[width=6in, height = 2.5in]{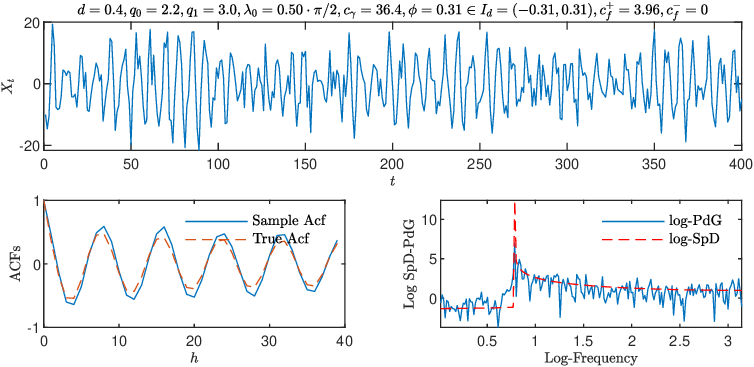}
          \vspace{-2mm}
      \caption{\textit{CLM realization with theoretical ACVF, sample ACVF, spectral density and periodogram.}}\label{f:fig6}  
  \end{center}
  \end{figure}

    \begin{figure}[!h]
    \vspace{-1mm}
    \begin{center}
    \includegraphics[trim={1cm 0 0cm 0}]{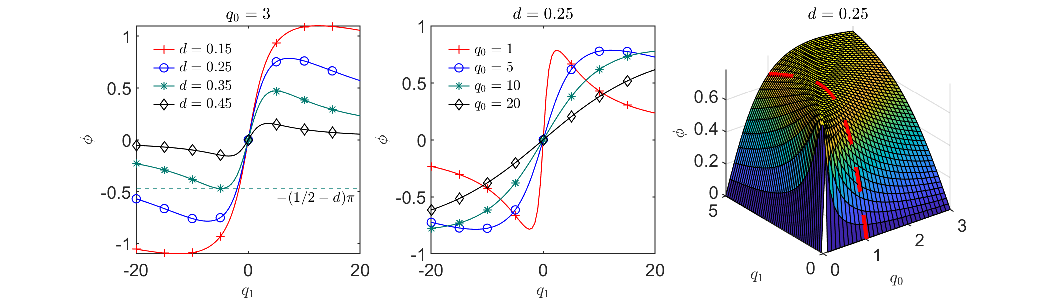}
          \vspace{-2mm}
      \caption{\textit{Relationship between $q_1$ and $\phi$ for fixed $q_0$ (left) and
      for fixed $d$ (middle and right).}}\label{f:fig7}  
  \end{center}
  \end{figure}

\section{Conclusions}
\label{sec:conclusions}

We have introduced a way to construct time series exhibiting CLM based on random modulation (RMod). We have also provided a class of parametric models capturing general CLM, admitting explicit autocovariance functions, spectral densities, and linear representations with respect to a white noise sequence. A crucial element of these constructions is the decoupling of the quasi-periodicity and long memory (LM) that appear in the autocovariance functions of the CLM series.

The constructed parametric RMod series, on the one hand, significantly extend the basic model that was considered in \cite{ProMad22} and, on the other hand, rely on delicate constructions of bivariate parametric LM series considered in \cite{KecPi15}. We mention here two advantages compared to other parametric CLM models: first, our model allows for more flexibility with regard to modeling, due to the presence of a cyclical phase $\phi$ and a careful analysis of the model when $\phi$ is on the ``boundary" of its admissible set; second, the construction of the RMod series allows for explicit calculations of ACVF and spectra that are often not available for other CLM models, such as the Gegenbauer series. These explicit quantities can be used readily for downstream tasks, such as simulation.

The flexibility in the modeling of CLM with RMod series is also evidenced by the extensions that were undertaken in Section \ref{sec:extensions}, including the case of multiple singularities in the spectrum and the ``boundary" case. We conjecture that it is also possible to consider RMod series in several different settings, e.g. multivariate RMod series (for the multivariate Gegenbauer processes see, e.g., \cite{WuPei18}) or RMod random fields (again, for the Gegenbauer counterpart, see, e.g., \cite{EspejoLeonenkoRuizMedina+2014+1+16}). It is also interesting to consider statistical tasks for RMod series, e.g., estimation for the location of the singularity $\lambda_0$ and the memory parameter.

\section*{Acknowledgments}
VP was supported in part by the ONR grants N00014-19-1-2092 and N00014-23-1-2176. PZ was partly supported by a dissertation completion fellowship that was extended from UNC's Graduate School.

 \bibliographystyle{abbrv}
\bibliography{cas-refs}

\begin{thebibliography}{10}

\bibitem{Anděl86}
J.~Anděl.
\newblock Long memory time series models.
\newblock {\em Kybernetika}, 22(2):105--123, 1986.

\bibitem{phdArt98}
J.~Arteche.
\newblock {\em Seasonal and cyclical long-memory in time series}.
\newblock PhD thesis, London School of Economics and Political Science, 1998.

\bibitem{Arteche20002}
J.~Arteche.
\newblock Gaussian semiparametric estimation in seasonal/cyclical long memory
  time series.
\newblock {\em Kybernetika}, 36(3):279--310, 2000.

\bibitem{Arteche2000}
J.~Arteche.
\newblock Log-periodogram regression in asymmetric long memory.
\newblock {\em Kybernetika}, 36(4):415--435, 2000.

\bibitem{arteche1999}
J.~Arteche and P.~M. Robinson.
\newblock Seasonal and cyclical long memory.
\newblock In S.~Ghosh, editor, {\em Asymptotics, Nonparametrics, and Time
  Series}, volume 158 of {\em Statistics Textbooks and Monographs}, pages
  115--148. CRC Press, Boca Raton, first edition edition, 1999.

\bibitem{Arteche20003}
J.~Arteche and P.~M. Robinson.
\newblock Semiparametric inference in seasonal and cyclical long memory
  processes.
\newblock {\em Journal of Time Series Analysis}, 21(1):1--25, 2000.

\bibitem{BerFenGhoSuc13book}
J.~Beran, Y.~Feng, S.~Ghosh, and R.~Kulik.
\newblock {\em Long-{M}emory {P}rocesses}.
\newblock Springer, Heidelberg, 2013.
\newblock Probabilistic properties and statistical methods.

\bibitem{bingham_goldie_teugels_1987}
N.~H. Bingham, C.~M. Goldie, and J.~L. Teugels.
\newblock {\em Regular Variation}.
\newblock Encyclopedia of Mathematics and its Applications. Cambridge
  University Press, 1987.

\bibitem{DenPie54}
M.~S. Denis and W.~J. Pierson.
\newblock On the motions of ships in confused seas: Abstract.
\newblock In {\em SNAME Transcations}, volume~61, pages 280--357, 1954.

\bibitem{DisPeiPro16}
G.~S. Dissanayake, M.~S. Peiris, and T.~Proietti.
\newblock State space modeling of {G}egenbauer processes with long memory.
\newblock {\em Computational Statistics and Data Analysis}, 100:115--130, 2016.

\bibitem{Dis18}
G.~S. Dissanayake, M.~S. Peiris, and T.~Proietti.
\newblock {Fractionally differenced Gegenbauer processes with long memory: A
  review}.
\newblock {\em Statistical Science}, 33(3):413--426, 2018.

\bibitem{Erftri51}
A.~Erd{\'e}lyi and F.~G. Tricomi.
\newblock {The asymptotic expansion of a ratio of gamma functions.}
\newblock {\em Pacific Journal of Mathematics}, 1(1):133--142, 1951.

\bibitem{Espleooleand15}
R.~M. Espejo, N.~N. Leonenko, A.~Olenko, and M.~D. Ruiz-Medina.
\newblock On a class of minimum contrast estimators for {G}egenbauer random
  fields.
\newblock {\em TEST}, 24(4):657--680, 2015.

\bibitem{EspejoLeonenkoRuizMedina+2014+1+16}
R.~M. Espejo, N.~N. Leonenko, and M.~D. Ruiz-Medina.
\newblock Gegenbauer random fields.
\newblock {\em Random Operators and Stochastic Equations}, 22(1):1--16, 2014.

\bibitem{GiHiRo01}
L.~Giraitis, J.~Hidalgo, and P.~M. Robinson.
\newblock Gaussian estimation of parametric spectral density with unknown pole.
\newblock {\em The Annals of Statistics}, 29(4):987--1023, 2001.

\bibitem{GirKouSur12Book}
L.~Giraitis, H.~L. Koul, and D.~Surgailis.
\newblock {\em Large Sample Inference for Long Memory Processes}.
\newblock Imperial College Press, London, 2012.

\bibitem{GirLei95}
L.~Giraitis and R.~Leipus.
\newblock A generalized fractionally differencing approach in long-memory
  modeling.
\newblock {\em Lithuanian Mathematical Journal}, 35(1):65--81, 1995.

\bibitem{GraZhaWoo89}
H.~L. Gray, N.-F. Zhang, and W.~A. Woodward.
\newblock On generalized fractional processes.
\newblock {\em Journal of Time Series Analysis}, 10(3):233--257, 1989.

\bibitem{GraZhaWoo94}
H.~L. Gray, N.-F. Zhang, and W.~A. Woodward.
\newblock On generalized fractional processes – a correction.
\newblock {\em Journal of Time Series Analysis}, 15(5):561--562, 1994.

\bibitem{Hos81}
J.~R.~M. Hosking.
\newblock Fractional differencing.
\newblock {\em Biometrika}, 68(1):165--176, 1981.

\bibitem{IvaLeo13}
A.~V. Ivanov, N.~Leonenko, M.~D. Ruiz-Medina, and I.~N. Savich.
\newblock Limit theorems for weighted nonlinear transformations of {G}aussian
  stationary processes with singular spectra.
\newblock {\em The Annals of Probability}, 41(2):1088--1114, 2013.

\bibitem{KecPi15}
S.~Kechagias and V.~Pipiras.
\newblock Definitions and representations of multivariate long-range dependent
  time series.
\newblock {\em Journal of Time Series Analysis}, 36(1):1--25, 2015.

\bibitem{Lew89}
E.~V. Lewis.
\newblock {\em Principles of Naval Architecture. Volume III. Motions in Waves
  and Controllability}.
\newblock Society of Naval Architects and Marine Engineers (U.S.), 1989.

\bibitem{Lonseldea57}
M.~S. Longuet-Higgins.
\newblock The statistical analysis of a random, moving surface.
\newblock {\em Philosophical Transactions of the Royal Society of London.
  Series A, Mathematical and Physical Sciences}, 249(966):321--387, 1957.

\bibitem{Maddanu:2022}
F.~Maddanu.
\newblock A harmonically weighted filter for cyclical long memory processes.
\newblock {\em AStA Advances in Statistical Analysis}, 106(1):49--78, 2022.

\bibitem{MadPro22}
F.~Maddanu and T.~Proietti.
\newblock Modelling persistent cycles in solar activity.
\newblock {\em Solar Physics}, 297(13):105299, 2022.

\bibitem{Tuc12}
T.~S. McElroy and S.~H. Holan.
\newblock On the computation of autocovariances for generalized {G}egenbauer
  processes.
\newblock {\em Statistica Sinica}, 22(4):1661--1687, 2012.

\bibitem{Ole13}
A.~Olenko.
\newblock Limit theorems for weighted functionals of cyclical long-range
  dependent random fields.
\newblock {\em Stochastic Analysis and Applications}, 31(2):199--213, 2013.

\bibitem{Oul02}
M.~Ould~Haye.
\newblock Asymptotic behavior of the empirical process for {G}aussian data
  presenting seasonal long-memory.
\newblock {\em ESAIM Probability and Statistics}, 6:293--309, 2002.

\bibitem{OulPhi11}
M.~{Ould Haye} and A.~Philippe.
\newblock Marginal density estimation for linear processes with cyclical long
  memory.
\newblock {\em Statistics and Probability Letters}, 81(9):1354--1364, 2011.

\bibitem{Pap83}
A.~Papoulis.
\newblock Random modulation: A review.
\newblock {\em IEEE Transactions on Acoustics, Speech, and Signal Processing},
  31(1):96--105, 1983.

\bibitem{PapBook}
A.~Papoulis.
\newblock {\em Probability, Random Variables, and Stochastic Processes}.
\newblock McGraw-Hill series in electrical engineering. McGraw-Hill, third
  edition, 1991.

\bibitem{pipiras_taqqu_2017}
V.~Pipiras and M.~S. Taqqu.
\newblock {\em Long-Range Dependence and Self-Similarity}.
\newblock Cambridge Series in Statistical and Probabilistic Mathematics.
  Cambridge University Press, 2017.

\bibitem{Pipzousap22}
V.~Pipiras, P.~Zoubouloglou, and T.~Sapsis.
\newblock Cyclical long memory in ship motions at non-zero speed.
\newblock In J.~M. Tomasz~Hinz, Przemysław~Krata, editor, {\em Proceedings of
  the 18th International Ship Stability Workshop (Gdańsk, Poland, September
  2022)}, pages 297--306, 2022.

\bibitem{ProMad22}
T.~Proietti and F.~Maddanu.
\newblock Modelling cycles in climate series: The fractional sinusoidal
  waveform process.
\newblock {\em Journal of Econometrics}, 239(1):105299, 2024.

\bibitem{SamBook16}
G.~Samorodnitsky.
\newblock {\em Stochastic Processes and Long Range Dependence}.
\newblock Springer Series in Operations Research and Financial Engineering.
  Springer, Cham, 2016.

\bibitem{viadenopp95}
M.~C. Viano, C.~Deniau, and G.~Oppenheim.
\newblock Long-range dependence and mixing for discrete time fractional
  processes.
\newblock {\em Journal of Time Series Analysis}, 16(3):323--338, 1995.

\bibitem{WooCheGra98}
W.~A. Woodward, Q.~C. Cheng, and H.~L. Gray.
\newblock A {$k$}-factor {GARMA} long-memory model.
\newblock {\em Journal of Time Series Analysis}, 19(4):485--504, 1998.

\bibitem{WuPei18}
H.~Wu and S.~Peiris.
\newblock An introduction to vector {G}egenbauer processes with long memory.
\newblock {\em Stat}, 7(1):e197, 2018.

\end{thebibliography}

\vspace{\baselineskip}



\newpage

\appendix

\section{Definitions of CLM in Time and Spectral Domains} \label{app:sec-defn}
\normalsize

In this appendix, we formalize the definitions of CLM in the spectral and time domains, incorporating both the ``non-boundary" and ``boundary" cases. In Appendix \ref{app:sec-spec-to-time} (resp. \ref{app:sec-time-to-spec}) below, we provide some conditions for a series exhibiting CLM in the spectral domain (resp. time domain) to also exhibit CLM in the time domain (resp. spectral domain). Finally, we prove a technical lemma that will be useful for the ``boundary" case (of Appendix \ref{app:subsec-time-to-spec-boundary}) in Appendix \ref{app:sec-auxillary} below. 

For clarity, Definition \ref{defn:clm-spec} below is given for series exhibiting CLM with one singularity, but analogues can be developed in the case of multiple singularities.


\begin{defn}(Spectral Domain) \label{defn:clm-spec}
    A second-order stationary time series $\{X_n\}_{n \in \ZZ}$ is said to exhibit CLM if its spectral density satisfies
\begin{equation} \label{eq:spec-rigorous-defn}
    f_X(\lambda) = \begin{cases}
        \left(c_f^- + R_f^-(\lambda_0 - \lambda) \right) (\lambda_0 - \lambda)^{-2d}, & 0 < \lambda < \lambda_0, \vspace{3mm}\\
        \left((c_f^+ + R_f^+(\lambda - \lambda_0)  \right) (\lambda - \lambda_0)^{-2d}, & \lambda_0 < \lambda < \pi,
    \end{cases}
\end{equation}
where $\lambda_0 \in (0,\pi), d \in \left( 0, \frac{1}{2} \right)$, $c_f^\pm \ge 0$ with $c_f^+ + c_f^- > 0$, and $R_f^\pm: (0,\infty) \to [-c_f^\pm,\infty)$ are functions with $R_f^\pm(x) \to 0$, as $x \to 0^+$. If, in addition, $c_f^+ = 0$ or $c_f^- = 0$, then we say that $\{X_n\}_{n \in \ZZ}$ exhibits CLM in the ``boundary" case. The ``non-boundary" case corresponds to $c_f^+ > 0$ and $c_f^- > 0$.
\end{defn}

The next definition characterizes CLM in the time domain.

\begin{defn}(Time domain) \label{defn:clm-time}
    A second-order stationary time series $\{X_n\}_{n \in \ZZ}$ is said to exhibit CLM if its autocovariance function satisfies 
    \begin{align} \label{eq:def-eq-acvf} 
       \gamma_X(h) &= \left( c_\gamma \cos(\lambda_0 h + \phi) + R_\gamma(h)   \right) h^{2d-1},  \quad h \in \NN_0, 
    \end{align}
    where $ \lambda_0 \in (0,\pi), d \in (0,1/2), c_\gamma \in (0,\infty), \phi \in [-\left(\frac{1}{2} - d\right) \pi, \left(\frac{1}{2} - d\right) \pi]$,  and $R_\gamma: [0,\infty) \to \RR$ is a function with $R_\gamma(x) \to 0, \quad \text{as } x \to + \infty$. If, in addition, $\phi = -\left(\frac{1}{2} - d\right) \pi$ or $\phi = \left(\frac{1}{2} - d\right) \pi$, then we say that $\{X_n\}_{n \in \ZZ}$ exhibits CLM in the ``boundary" case. The ``non-boundary" case is $\phi \in (-\left(\frac{1}{2} - d\right) \pi, \left(\frac{1}{2} - d\right) \pi)$.
\end{defn}

The next remark provides an alternative useful formulation for CLM in the time domain. 

\begin{remark} \label{rmk-clm-defn-time-equiv}
   The relation \eqref{eq:def-eq-acvf} can be recast as
    \begin{equation} \label{def-time-2}
    \gamma_X(h) = c_{1,\gamma} \cos(\lambda_0 h)  h^{2d-1} + c_{2,\gamma} \sin(\lambda_0 h)  h^{2d-1} + R_\gamma(h) h^{2d-1},
    \end{equation}
    where $R_\gamma$ is as in Definition \ref{defn:clm-time}, and
    \begin{equation} \label{eq:c1g-c2g-cg}
     c_{1,\gamma} \doteq c_\gamma \cos(\phi) \in (0,\infty), \quad c_{2,\gamma} \doteq - c_\gamma \sin(\phi) \in \RR.
    \end{equation}
\end{remark}

\begin{remark}
    Definitions \ref{defn:clm-spec} and \ref{defn:clm-time} suggest the relationship $c_f^\mp = 0 \Leftrightarrow \phi = \pm \left(\frac{1}{2} - d\right) \pi$ in the ``boundary" case. This will be discussed in Appendix \ref{subsubsec:admis-phase} below.
\end{remark}

\begin{remark}
    More general definitions of CLM could be based on using slowly varying functions instead of constants $c_f^\pm$ in \eqref{eq:spec-rigorous-defn} and $c_{i,\gamma}$ in \eqref{def-time-2}. We restrict our analysis to the case of constants, which is more relevant to modeling in practice. The role of functions $R_f^\pm$ in \eqref{eq:spec-rigorous-defn} is to allow for deviations from the constants $c_f^\pm$, and Definition \ref{defn:clm-spec} naturally captures the divergence of the spectral density around frequency $\lambda_0$. By having the function $R_\gamma(h)$ in \eqref{eq:def-eq-acvf} or \eqref{def-time-2}, we aimed to have a general definition of CLM in the time domain. There may, however, be ways to make it more inclusive. For example, take two uncorrelated series $X$ and $Y$ with $X$ being CLM in the sense of \ref{defn:clm-time} and $Y$ being LM with the memory parameter $\tilde d > d$. One may want to call $X+Y$ CLM, but $X+Y$ is not CLM in the sense of Definition \ref{defn:clm-time}, since the LM decay $h^{2\tilde d-1}$ associated with $Y$ cannot be incorporated in the remainder term $R_\gamma(h)h^{2d-1}$ in \eqref{eq:def-eq-acvf} .
\end{remark}

We note that series exhibiting \textit{Seasonal Cyclical Asymmetric Long Memory} (SCALM) presented in, e.g., \cite{Arteche20002,Arteche2000}, also exhibit CLM in the boundary case (i.e., $c_f^+ = 0$ or $c_f^- = 0$) according to Definition \ref{defn:clm-spec}.

Definitions \ref{defn:clm-spec} and \ref{defn:clm-time} are, in general, not equivalent. We are aware of only two works relating some versions of these two definitions. First, in Proposition 2 of \cite{viadenopp95}, an argument is provided for passing from the spectral to the time domain definition for extended fractional ARMA processes with seasonal effects in the ``non-boundary" case. However, even for this specific class of processes, the reader is referred to a different, but similar proof, and so the arguments are not complete. Second, Lemma 1 of Chapter 2 in \cite{phdArt98} considers an argument for passing from the spectral to the time domain definition in the ``boundary" case. While the proof strategy is valid (and, in fact, similar to the one used in Appendix \ref{app:subsec-spec-to-time-boundary}), it ignores second-order asymptotic expansions (see Appendix \ref{app:sec-auxillary}), thus rendering the stated result imprecise. We are not aware of any results obtaining a spectral domain representation from the time domain. 

In relating Definitions \ref{defn:clm-spec} and \ref{defn:clm-time} we shall use slowly varying functions converging to a constant, but also having the following quasi-monotonicity property. A slowly varying function $L:[0, \infty) \to (0,\infty)$ at infinity is called \textit{quasi-monotone} (see, e.g., Chapter 2.7 in \cite{bingham_goldie_teugels_1987}) if the following two conditions hold: $(i)$ it is of bounded variation on any compact interval of $[0,\infty)$ and $(ii)$ for some $\delta > 0$, $
\int_0^x u^\delta |d L(u)| = O(x^\delta L(x)),$ as $x \to \infty$.

\section{From Spectral to Time Domain} \label{app:sec-spec-to-time}

When going from the spectral to the time domain, the terms $c_f^-(\lambda_0 - \lambda)^{-2d}$ and $c_f^+(\lambda - \lambda_0)^{-2d}$ in Definition \ref{defn:clm-spec} will lead to $c_\gamma \cos(\lambda_0 h + \phi) h^{2d-1}$ in Definition \ref{defn:clm-time}, plus a remainder of order $h^{-1}$. The ultimate form of the remainder $R_\gamma(h)h^{2d-1}$ in \eqref{eq:def-eq-acvf} will be determined by the exact forms of $R_f^-(\lambda_0-\lambda)$ and $R_f^+(\lambda-\lambda_0)$. Various forms are possible. Results below provide a few examples and illustrations. 

In Appendix \ref{app:subsec-spec-to-time-non-boundary}, we state and prove a result in the ``non-boundary" case. We emphasize that special treatment is required for the ``boundary" case, and state a related result in Appendix \ref{app:subsec-spec-to-time-boundary}. Finally, we derive the set of admissible cyclical phases and note some special cases for the cyclical phase $\phi$ in Appendix \ref{subsubsec:admis-phase}.

\subsection{``Non-boundary" Case} \label{app:subsec-spec-to-time-non-boundary}

The following proposition provides sufficient conditions for Definition \ref{defn:clm-spec} to imply Definition \ref{defn:clm-time} in the ``non-boundary" case.

\begin{proposition} \label{prop:equiv-defn-spec-to-time} 
    Let
    \begin{equation}
        R_f^\pm(x) = L_f^\pm\left(\frac{1}{x}\right) - c_f^\pm,
    \end{equation}
    where $L_f^-: \left(\frac{1}{\lambda_0},\infty\right) \to (0,\infty)$ and $L_f^+: \left(\frac{1}{\pi - \lambda_0},\infty\right) \to (0,\infty)$ are  quasi-monotone slowly varying functions at $+ \infty$, with $L_f^\pm(x) \sim c_f^\pm \in (0,\infty)$ as $x \to + \infty$. Then, Definition \ref{defn:clm-spec} implies Definition \ref{defn:clm-time} with \eqref{eq:cgamma-psi}.
\end{proposition}

\begin{proof}
We have
\begin{equation} \label{eq:210}
\begin{split}
\gamma_X(h) &=   2\int_{0}^{\pi} \cos(h\lambda) f_X(\lambda) d\lambda  = 2 \int_0^{\lambda_0} \cos(h\lambda) f_X(\lambda) d\lambda + 2 \int_{\lambda_0}^{\pi} \cos(h\lambda) f_X(\lambda) d\lambda \\
&\doteq 2\left[ \gamma_-(h) + \gamma_+(h)   \right].
\end{split}
\end{equation}
We consider these two quantities separately. First,
\begin{equation} \label{eq:spec-to-time-gamma-}
\begin{split}
    \gamma_-(h) &= \int_0^{\lambda_0} \cos(h \lambda ) L_f^-\left( \frac{1}{\lambda_0 - \lambda}\right) (\lambda_0 - \lambda)^{-2d}   d\lambda = \int_0^{\lambda_0} \cos(h (\lambda_0 - \om)) L_f^-\left( \frac{1}{\om} \right) \om^{-2d}   d\om \\
    &= \cos(h \lambda_0) \int_0^{\lambda_0} \cos(h \om)  L_f^-\left( \frac{1}{\om} \right) \om^{-2d}   d\om + \sin(h \lambda_0) \int_0^{\lambda_0} \sin(h \om)  L_f^-\left( \frac{1}{\om} \right) \om^{-2d}   d\om,
\end{split}
\end{equation}
where the second equality follows from the change of variables $\lambda_0 - \lambda = \om$. By Proposition A.2.2 of \cite{pipiras_taqqu_2017} and since $L_f^-$ is quasi-monotone, we have that, as $h\to\infty$,
\begin{equation}
    \int_0^{\lambda_0} \cos(h \om)  L_f^-\left( \frac{1}{\om} \right) \om^{-2d}   d\om = h^{2d-1} L_f^-(h) \Gamma(1-2d) \sin(\pi d) + o(h^{2d-1}),
\end{equation}
and similarly
\begin{equation}
    \int_0^{\lambda_0} \sin(h \om)  L_f^-\left( \frac{1}{\om} \right) \om^{-2d}   d\om  = h^{2d-1} L_f^-(h) \Gamma(1-2d) \cos(\pi d) + o(h^{2d-1}).
\end{equation}
Analogous calculations show that
\begin{equation} \label{eq:spec-to-time-gamma+}
\begin{split}
    \gamma_+(h) &= \int_{\lambda_0}^\pi \cos(h \lambda ) L_f^+\left( \frac{1}{\lambda - \lambda_0}\right) ( \lambda - \lambda_0)^{-2d}   d\lambda = \int_0^{\pi - \lambda_0} \cos(h (\lambda_0 + \om)) L_f^+\left( \frac{1}{\om} \right) \om^{-2d}   d\om \\
    &= \cos(h \lambda_0) \int_0^{\pi - \lambda_0} \cos(h \om)  L_f^+\left( \frac{1}{\om} \right) \om^{-2d}   d\om - \sin(h \lambda_0) \int_0^{\pi - \lambda_0} \sin(h \om)  L_f^+\left( \frac{1}{\om} \right) \om^{-2d}   d\om,
\end{split}
\end{equation}
where
\begin{equation}
    \int_0^{\pi - \lambda_0} \cos(h \om)  L_f^+\left( \frac{1}{\om} \right) \om^{-2d}   d\om = h^{2d-1} L_f^+(h) \Gamma(1-2d) \sin\left( \pi d   \right) + o(h^{2d-1}),
\end{equation}
and
\begin{equation}
    \int_0^{\pi - \lambda_0} \sin(h \om)  L_f^+\left( \frac{1}{\om} \right) \om^{-2d}   d\om  = h^{2d-1} L_f^+(h) \Gamma(1-2d) \cos\left( \pi d   \right) + o(h^{2d-1}).
\end{equation}
Since $L_f^{\pm}(h) \to c_f^{\pm}$ as $h \to \infty$, \eqref{eq:210} can be written as
\begin{equation}
\begin{split}
    \gamma_X(h) &= \cos(h \lambda_0) \left[ h^{2d-1} c_f^- 2 \Gamma(1-2d) \sin(\pi d)  +R_1^-(h) \right] \\
    &\quad+ \sin(h \lambda_0) \left[ h^{2d-1} c_f^- 2 \Gamma(1-2d) \cos(\pi d)  + R_2^-(h) \right] \\
    &\quad + \cos(h \lambda_0) \left[ h^{2d-1} c_f^+ 2 \Gamma(1-2d) \sin(\pi d)  + R_1^+(h) \right] \\
    &\quad- \sin(h \lambda_0) \left[ h^{2d-1} c_f^+ 2 \Gamma(1-2d) \cos(\pi d)  + R_2^+(h) \right],
\end{split}
\end{equation}
where $R_{1}^\pm(h) \doteq 2(L_f^\pm(h) - c_f^{\pm}) \Gamma(1-2d) \sin(\pi d) h^{2d-1} + o(h^{2d-1})$ and $R_{2}^\pm(h) \doteq 2(L_f^\pm(h) - c_f^{\pm}) \Gamma(1-2d) \cos(\pi d) h^{2d-1} + o(h^{2d-1})$. We thus have that
\begin{equation} \label{eq:213}
\begin{split}
    \gamma_X(h) &= 2 \Gamma(1-2d) h^{2d-1}  \left[(c_f^+ + c_f^-)  \sin(\pi d) \cos(\lambda_0 h) + (c_f^- - c_f^+)  \cos(\pi d) \sin(\lambda_0h)\right] + R_\gamma(h) h^{2d-1} \\
    &= c_{\gamma} \cos(\lambda_0 h + \phi) h^{2d-1}  + R_\gamma(h) h^{2d-1},
\end{split}
\end{equation}
where $c_\gamma,\phi$ are given in \eqref{eq:cgamma-psi}, and 
\begin{equation} 
R_\gamma(h) h^{2d-1} \doteq \cos(h \lambda_0) R_1^-(h) + \sin(h \lambda_0) R_2^-(h) + \cos(h \lambda_0) R_1^+(h) - \sin(h \lambda_0) R_2^+(h) = o(h^{2d-1}).
\end{equation}
This concludes the proof.
\end{proof}

\begin{remark} \label{rmk:cg-c1g-c2g}
    Note that \eqref{eq:213} can be reformulated as in \eqref{def-time-2} with
    \begin{equation} \label{eq:cg1-cg2}
    c_{1,\gamma} \doteq 2\Gamma(1-2d) (c_f^- + c_f^+) \sin(\pi d ), \quad c_{2,\gamma} \doteq 2\Gamma(1-2d) (c_f^- - c_f^+) \cos(\pi d ).
   \end{equation}
\end{remark}



\subsection{``Boundary" Case} \label{app:subsec-spec-to-time-boundary}

As stated in Definition \ref{defn:clm-spec}, the ``boundary" case corresponds to $c_f^+ = 0$ or $c_f^- = 0$. For simplicity, we consider here only the case $c_f^- = 0$, leading to a phase $\phi = \left( \frac{1}{2} - d \right) \pi$. An analogous result can be stated for the case $c_f^+ = 0$, corresponding to $\phi = -\left( \frac{1}{2} - d \right) \pi$.

\begin{proposition} \label{it:spec-to-time-border} 
    Let 
    \[
    c_f^- = 0, \quad R_f^-(x) = L_f^- \left( \frac{1}{x}  \right) x^{2\veps},\quad  R_f^+(x) = L_f^+ \left( \frac{1}{x} \right) - c_f^+,
    \]
    where $L_f^-: \left(\frac{1}{\lambda_0},\infty\right) \to (0,\infty)$, $L_f^+: \left(\frac{1}{\pi - \lambda_0},\infty\right) \to (0,\infty)$ are two quasi-monotone slowly varying functions at $+\infty$, with $L_f^+(x) \sim c_f^+$ as $x \to \infty$, and $\veps \in (0,d)$. Then, Definition \ref{defn:clm-spec} implies Definition \ref{defn:clm-time} in the ``boundary" case with
    \begin{equation} \label{eq:cgamma-phi-boundary}
        c_\gamma \doteq 2\Gamma(1-2d) c_f^+, \quad \phi \doteq \left( \frac{1}{2} - d \right) \pi.
    \end{equation}
\end{proposition}

\begin{proof}
Recall the relation \eqref{eq:210}. In view of \eqref{eq:spec-to-time-gamma+} and the form of $R_f^+$ in the assumption, we write
\begin{equation} \label{eq:bound-g+}
\begin{split}
   2\gamma_+(h) &= 2 \cos(h \lambda_0) \int_0^{\pi - \lambda_0} \cos(h  \om )  L_{f}^+\left(\frac{1}{\om}\right) \om^{-2d } d\om \\
   &\quad- 2\sin(h \lambda_0) \int_0^{\pi - \lambda_0} \sin(h  \om )  L_{f}^+\left(\frac{1}{\om}\right) \om^{-2d } d\om \\
   &= 2 \Gamma(1-2d) \sin(d \pi)  L_{f}^+\left(h \right) \cos(\lambda_0 h) h^{2d-1 } -  2 \Gamma(1-2d ) \cos(d \pi)  L_{f}^+\left(h \right) \sin(\lambda_0 h) h^{2d-1 } \\
   &\quad+ o(h^{2d-1})  \\
    &= 2 \Gamma(1-2d) \sin(d \pi)  c_f^+ \cos(\lambda_0 h) h^{2d-1 } -  2 \Gamma(1-2d ) \cos(d \pi)  c_f^+ \sin(\lambda_0 h) h^{2d-1 } + o(h^{2d-1})  \\
     &= c_\gamma \cos(\lambda_0 h + \phi) h^{2d-1} + o(h^{2d-1}),
\end{split}
\end{equation}
where the first equality follows from Proposition A.2.2 of \cite{pipiras_taqqu_2017}, and $\phi,c_\gamma$ are defined in \eqref{eq:cgamma-phi-boundary}.

Now we compute the contribution of $\gamma_-$ in \eqref{eq:spec-to-time-gamma-}. From Proposition A.2.2 of \cite{pipiras_taqqu_2017} and the form of $R_f^-$, as $h \to \infty$,
\begin{equation} \label{eq:bound-g-}
\begin{split}
2\gamma_-(h) &= 2\cos(h \lambda_0) \int_0^{\lambda_0} \cos(h \lambda ) R_{f}^-\left( \om \right) \om^{-2 d} d\om + 2\sin(h \lambda_0) \int_0^{\lambda_0} \sin(h \om )  R_{f}^-\left( \om \right) \om^{-2d} d\om  \\
&= 2\cos(h \lambda_0) \int_0^{\lambda_0} \cos(h \lambda ) L_{f}^-\left( \frac{1}{\om} \right) \om^{-2 d + 2\veps } d\om + 2\sin(h \lambda_0) \int_0^{\lambda_0} \sin(h \om )  L_{f}^-\left( \frac{1}{\om} \right) \om^{-2d + 2\veps} d\om \\
    &= 2\Gamma(1-2d + 2\veps) L_{f}^-(h) \left(\cos(\pi (d - \veps )) \sin(h \lambda_0 )+ \sin(\pi (d - \veps )) \cos(h \lambda_0) \right) h^{2d-1 - 2\veps} \\
    &\quad+ o(h^{2(d-\veps)-1}) = o(h^{2d-1}) .
\end{split}
\end{equation}
Combining \eqref{eq:210}, \eqref{eq:bound-g+}, and \eqref{eq:bound-g-}, we write
\begin{equation}
    \gamma_X(h) = c_\gamma \cos(\lambda_0 h + \phi) h^{2d-1} + o(h^{2d-1}), 
\end{equation}
where $c_\gamma,\phi$ are defined in \eqref{eq:cgamma-phi-boundary}. 
\end{proof}

\section{From Time to Spectral Domain} \label{app:sec-time-to-spec}

When going from the time to the spectral domain, the term $c_\gamma \cos(\lambda_0 h + \phi) h^{2d-1}$ in Definition \ref{defn:clm-time} will lead to $c_f^- (\lambda_0 - \lambda)^{-2d}$, $c_f^+ (\lambda - \lambda_0)^{-2d}$ in Definition \ref{defn:clm-time}, plus a remainder. The ultimate remainders $R_f^\pm$ in Definition \ref{defn:clm-time} will be determined by the exact form of $R_\gamma(h)$.

In Appendix \ref{app:subsec-time-to-spec-non-boundary}, we state and prove a result in the ``non-boundary" case $\phi \in \left( -\left( \frac{1}{2} -d \right) \pi, \left( \frac{1}{2} -d \right) \pi \right)$. A special treatment is required for the ``boundary" case $\phi = \pm \left( \frac{1}{2} -d \right) \pi$, and a related result is stated in Appendix \ref{app:subsec-time-to-spec-boundary}. Recall the sign function in \eqref{def:sign-func}.

\subsection{``Non-boundary" Case} \label{app:subsec-time-to-spec-non-boundary}

\begin{proposition} \label{prop:equiv-defn-time-to-spec} 
    Assume that $\phi \in \left(\left(d - \frac{1}{2} \right)\pi  , \left(\frac{1}{2} - d\right) \pi \right)$, and that for $h \ge 1$,
    \begin{equation} \label{eq:time-to-spec-rgamma}
        R_\gamma(h) \doteq (L_{1,\gamma}(h) - c_{\gamma} \cos(\phi)) \cos(\lambda_0 h)  + \operatorname{sign}(-\sin(\phi)) (L_{2,\gamma}(h) - |c_{\gamma} \sin(\phi) |) \sin(\lambda_0 h) ,
    \end{equation}
    where $L_{1,\gamma}, L_{2,\gamma} : (0,\infty) \to (0,\infty)$ are quasi-monotone slowly varying functions at $+\infty$, with $L_{1,\gamma}(x) \sim c_{\gamma} \cos(\phi) \in (0,\infty)$ and for $\phi \neq 0$, $L_{2,\gamma}(x) \sim |c_{\gamma} \sin(\phi) | \in (0,\infty)$ as $x \to +\infty$.  Then, Definition \ref{defn:clm-time} implies Definition \ref{defn:clm-spec} with \eqref{eq:cfpm-cgamma}.
\end{proposition}

\begin{proof}
As in the proof of Proposition 2.2.14 in Appendix A.2 of \cite{pipiras_taqqu_2017}, for $\lambda \in [0,\pi) \setminus \{\lambda_0\}$,
\begin{equation}
f_X(\lambda) =  \frac{1}{2\pi} \sum_{h=-\infty}^{\infty} e^{-ih\lambda} \gamma_X(h).
\end{equation}
Strictly speaking, the arguments below are first used to show that the spectral density can be written this way and then the asymptotics are established. In particular,
\begin{equation} \label{eq:456}
\begin{split}
        f_X(\lambda) &=  \frac{1}{2\pi} \left[ \gamma_X(0) + 2\sum_{h=1}^\infty \cos( h \lambda ) \left(L_{1,\gamma}(h) \cos( \lambda_0 h ) + \operatorname{sign}(-\sin(\phi))  L_{2,\gamma}(h) \sin( \lambda_0 h ) \right) h^{2d-1} \right] \\ 
    &= \frac{1}{2\pi} \left[ \gamma_X(0) + f_1(\lambda) - \operatorname{sign}(-\sin(\phi))  f_2(\lambda) +  f_3(\lambda) + \operatorname{sign}(-\sin(\phi)) f_4(\lambda) \right],
\end{split}
\end{equation}
where
\begin{equation} \label{eq:f1-f2-def}
    f_1(\lambda) \doteq \sum_{h=1}^\infty \cos\big((\lambda-\lambda_0) h \big) L_{1,\gamma}(h) h^{2d-1}, \quad f_2(\lambda) \doteq \sum_{h=1}^\infty \sin\big((\lambda-\lambda_0) h \big) L_{2,\gamma(h)} h^{2d-1},
\end{equation}
\begin{equation} \label{eq:f3-f4-def}
    f_3(\lambda) \doteq  \sum_{h=1}^\infty \cos((\lambda + \lambda_0)h  )  L_{1,\gamma}(h)  h^{2d-1}, \quad f_4(\lambda) \doteq  \sum_{h=1}^\infty \sin((\lambda + \lambda_0)h  )  L_{2,\gamma}(h)  h^{2d-1}.
\end{equation}
By Proposition A.2.1 of \cite{pipiras_taqqu_2017}, as $\lambda \to \lambda_0$,
\begin{equation} \label{eq:f1-f2}
\begin{split}
    f_1(\lambda) &= |\lambda-\lambda_0|^{-2d} L_{1,\gamma} \left(  \frac{1}{|\lambda - \lambda_0|}  \right) \Gamma(2d) \cos(\pi d) + o(|\lambda-\lambda_0|^{-2d}), \\
    f_2(\lambda) &= \text{sign}(\lambda - \lambda_0) |\lambda-\lambda_0|^{-2d} L_{2,\gamma} \left(  \frac{1}{|\lambda - \lambda_0|}  \right) \Gamma(2d) \sin(\pi d)+ o(|\lambda-\lambda_0|^{-2d}).
\end{split}
\end{equation}

On the other hand, fix some $0 < \alpha < \lambda_0 < \beta < \pi$ so that $\alpha - \lambda_0 < 0$, $\beta - \lambda_0 > 0$, and $\sup_{\hat \lambda \in [\alpha,\beta]} \{\frac{2}{\sin(\hat \lambda/2)},\frac{2}{\cos(\hat\lambda/2)}\} < \infty$. Since $L_{1,\gamma},L_{2,\gamma}$ are quasi-monotone, the relation (A.2.6) in Proposition A.2.1 of \cite{pipiras_taqqu_2017} says that, for $\lambda \in [\alpha-\lambda_0,\beta-\lambda_0]$,
\[
\left| \sum_{h=n}^\infty \cos( h (\lambda + \lambda_0 )) \frac{L_{1,\gamma}(h)}{h^{1-2d}} \right| \le \frac{2}{\sup_{\hat \lambda \in [\alpha,\beta]} \{\frac{2}{\sin(\hat \lambda/2)}\}} \frac{|L_{1,\gamma}(n)|}{n^{1-2d}} (1  + O(1)),
\]
which implies that, for all $\epsilon > 0$, there exists some $M \ge 1$ such that
\[
\sup_{\lambda \in [\alpha-\lambda_0,\beta-\lambda_0]} \left| \sum_{h=M}^\infty \cos( h (\lambda + \lambda_0 )) \frac{L_{1,\gamma}(h)}{h^{1-2d}} \right| < \epsilon.
\]
Likewise, we obtain
\[
\sup_{\lambda \in [\alpha-\lambda_0,\beta-\lambda_0]} \left| \sum_{h=M}^\infty \sin( h (\lambda + \lambda_0 )) \frac{L_{2,\gamma}(h)}{h^{1-2d}} \right| < \epsilon.
\]
By truncating these series, this shows that, as $\lambda \to \lambda_0$,
\begin{equation} \label{eq:f3-f4}
    f_3(\lambda) \to \sum_{h=1}^\infty \cos( 2 \lambda_0  h ) \frac{L_{1,\gamma}(h)}{h^{1-2d}} = O(1), \quad f_4(\lambda)\to \sum_{h=1}^\infty \sin( 2 \lambda_0  h ) \frac{L_{2,\gamma}(h)}{h^{1-2d}} = O(1).
\end{equation}
Now define
\begin{equation}
\begin{split}
    R(\lambda) &\doteq \frac{\Gamma(2d)}{2\pi} \left[ \left( L_{1,\gamma} \left(  \frac{1}{|\lambda - \lambda_0|}  \right) - c_\gamma \cos(\phi)  \right)\cos(d\pi) \right. \\
    &\quad\left.-  \text{sign}(\lambda-\lambda_0)\left(  \operatorname{sign}(-\sin(\phi)) L_{2,\gamma} \left(  \frac{1}{|\lambda - \lambda_0|}  \right) + c_\gamma \sin(\phi) \right) \sin(d\pi)   \right]  .
\end{split}
\end{equation}
Since $L_{1,\gamma}(x) \to c_\gamma \cos(\phi) ,  \operatorname{sign}(-\sin(\phi)) L_{2,\gamma}(x) \to - c_\gamma \sin(\phi)$ as $x \to +\infty$, it follows that $R(\lambda) = o(1)$. By combining \eqref{eq:456}, \eqref{eq:f1-f2} and \eqref{eq:f3-f4}, we write
\begin{equation} \label{eq:time-to-spec-density}
\begin{split}
    f_X(\lambda) &= \frac{\Gamma(2d)}{2\pi} \left[ L_{1,\gamma} \left(  \frac{1}{|\lambda - \lambda_0|}  \right) \cos(d\pi) \right.\\
    &\quad\quad\quad\quad-\left. \operatorname{sign}(-\sin(\phi))  \text{sign}(\lambda-\lambda_0) L_{2,\gamma} \left(  \frac{1}{|\lambda - \lambda_0|}  \right) \sin(d\pi)   \right]    | \lambda - \lambda_0 |^{-2d} \\
    &\quad+ \frac{\Gamma(2d)}{2\pi} \left( \gamma_X(0) +f_3(\lambda) + \operatorname{sign}(-\sin(\phi)) f_4(\lambda) \right) + o(|\lambda-\lambda_0|^{-2d}) \\ 
    &= \frac{\Gamma(2d)}{2\pi} \left[ c_\gamma \cos(\phi)    \cos(d\pi) + \text{sign}(\lambda-\lambda_0) c_{\gamma} \sin(\phi)  \sin(d\pi)   \right]    | \lambda - \lambda_0 |^{-2d} + R(\lambda) | \lambda - \lambda_0 |^{-2d} \\
    &\quad+  O(1) + o(|\lambda-\lambda_0|^{-2d}) \\ 
    &= \left( \frac{\Gamma(2d)}{2\pi} c_\gamma \cos(\pi d - \text{sign}(\lambda-\lambda_0)\phi) +  o(1) \right) 
 |\lambda - \lambda_0 |^{-2d},
    \end{split}
\end{equation}
where $c_\gamma,\phi$ were defined in \eqref{eq:cgamma-psi}. 
\end{proof}

\begin{remark}
    In view of Remark \ref{rmk:cg-c1g-c2g}, we can recast \eqref{eq:time-to-spec-rgamma} as
        \begin{equation}
        R_\gamma(h) = (L_{1,\gamma}(h) - c_{1,\gamma}) \cos(\lambda_0 h) h^{2d-1} + \operatorname{sign}(c_{2,\gamma}) (L_{2,\gamma}(h) - |c_{2,\gamma} |) \sin(\lambda_0 h) h^{2d-1},
    \end{equation}
    where $c_{1,\gamma},c_{2,\gamma}$ are given in \eqref{eq:c1g-c2g-cg}. Then, from \eqref{eq:cfpm-cgamma}, we have that
    \begin{equation} \label{eq:cf+--from-c1g-c2g}
        c_f^\pm \doteq \frac{\Gamma(2d)}{2\pi} \left( c_{1,\gamma} \cos( \pi d) \mp  c_{2,\gamma} \sin( \pi d) \right) \in (0,\infty).
    \end{equation}
\end{remark}

\subsection{``Boundary" Case} \label{app:subsec-time-to-spec-boundary}

For clarity, we shall focus on the case 
\begin{equation} 
    \phi =  \left( \frac{1}{2} -d \right) \pi,
\end{equation}
which is expected to correspond to $c_f^- = 0$. An analogous statement can be obtained in the case $\phi = - \left( \frac{1}{2} -d \right) \pi$, corresponding to $c_f^+ = 0$.

\begin{proposition} \label{prop:equiv-defn-time-to-spec-boundary} 
    Assume that $\phi = \left(\frac{1}{2} - d\right) \pi $, and that for $h \ge 1$,
    \begin{equation} \label{eq:time-to-spec-rgamma-bound}
        R_\gamma(h) \doteq L_{1,\gamma}(h)  \cos(\lambda_0 h) h^{-2\veps} + \xi L_{2,\gamma}(h)  \sin(\lambda_0 h) h^{-2\veps},
    \end{equation}
    where $\veps \in (0,d), \xi = \pm 1$, $L_{1,\gamma}, L_{2,\gamma} : (0,\infty) \to (0,\infty)$ are quasi-monotone slowly varying functions at $+\infty$, with $L_{1,\gamma}(x) \sim b_1 \in (0,\infty)$ and $\xi L_{2,\gamma}(x) \sim b_2 \in \RR$ as $x \to +\infty$, where 
    \begin{equation} \label{eq:b1-b2-cond}
        b_1 \in \RR \ \ \text{and}\ \ b_2 = 0,\quad \text{or} \quad \frac{b_1}{b_2} \in (-\tan(\pi (d-\veps)), \tan(\pi (d-\veps))). 
    \end{equation}
    Then, Definition \ref{defn:clm-time} implies Definition \ref{defn:clm-spec} with 
        \begin{equation} \label{prop:time-to-spec-bound-cf+}
    c_f^- = 0, \quad c_f^+ \doteq   \frac{c_\gamma}{2\pi} \Gamma(2d)\sin(2\pi d).
\end{equation}
\end{proposition}

\begin{proof}
We can write, for $h \in \ZZ$,
\begin{equation}
\gamma_X(h) = c_\gamma \cos(\lambda_0 h + \phi) h^{2d-1} + L_{1,\gamma}(h) \cos(\lambda_0 h) h^{2d-2\veps-1} + \xi L_{2,\gamma}(h) \sin(\lambda_0 h)h^{2d-2\veps-1} = \gamma_1(h) + \gamma_2(h),
\end{equation}
with
\begin{equation} 
\gamma_1(h) \doteq  c_\gamma \cos(\lambda_0 h + \phi) h^{2d-1}, \quad \gamma_2(h) \doteq L_{1,\gamma}(h) \cos(\lambda_0 h) h^{2d-2\veps-1} + \xi L_{2,\gamma}(h) \sin(\lambda_0 h) h^{2d-2\veps-1}.
\end{equation}
Calculations similar to \eqref{eq:456}--\eqref{eq:time-to-spec-density} and the fact that $\phi = \left( \frac{1}{2} - d \right) \pi$, imply that, as $\lambda \to \lambda_0^+$, 
\begin{equation} \label{eq:time-spec-boundary-101}
\begin{split}
    \frac{1}{2\pi}\sum_{h=-\infty}^\infty e^{-ih\lambda} \gamma_{1}(h) &= \left( \frac{\Gamma(2d)}{2\pi} c_\gamma \cos(\pi d - \text{sign}(\lambda-\lambda_0)\phi) +  o(1) \right) 
 |\lambda - \lambda_0 |^{-2d} \\
    &= \left( \frac{\Gamma(2d)}{2\pi} c_\gamma \sin(2\pi d) +  o(1) \right) 
 |\lambda - \lambda_0 |^{-2d}.
\end{split}
\end{equation}
This yields the formula for $c_f^+$ in \eqref{prop:time-to-spec-bound-cf+}.
On the other hand,  as $\lambda \to \lambda_0^-$,
\begin{equation} \label{eq:time-spec-boundary-100}
\begin{split}
  \frac{1}{2\pi}\sum_{h=-\infty}^\infty e^{-ih\lambda} \gamma_{1}(h) &= \frac{1}{2\pi} \bigg[ \gamma_1(0) + \sin(d\pi) c_\gamma\sum_{h=1}^\infty \cos((\lambda - \lambda_0) h) h^{2d-1} \\
  &\quad-  \cos(d \pi) c_\gamma \sum_{h=1}^\infty \sin((\lambda_0 - \lambda) h) h^{2d-1}\bigg. \\
  &\quad+ \bigg. \sin(d\pi) c_\gamma \sum_{h=1}^\infty \cos((\lambda + \lambda_0) h) h^{2d-1} - \cos(d\pi) c_\gamma \sum_{h=1}^\infty \sin((\lambda + \lambda_0) h) h^{2d-1} \bigg]  \\
&= \frac{1}{2\pi}\left[  \sin(d\pi) c_\gamma \sum_{h=1}^\infty \cos((\lambda + \lambda_0) h) h^{2d-1} \right.\\
&\quad- \left.   \cos(d\pi) c_\gamma \sum_{h=1}^\infty \sin((\lambda + \lambda_0) h) h^{2d-1} + O(1) \right]\\
& = O(1) ,
\end{split}
\end{equation}
where we used the form of $\phi$ and Lemma \ref{lemma:second-order-time-to-spec} to replace the two series (and $\gamma_1(0)$) by $O(1)$ in the second equality, and that the series in the fourth and fifth lines are $O(1)$ in the last equality, which follows from the same arguments as in \eqref{eq:f3-f4}. The relation \eqref{eq:time-spec-boundary-100} is consistent with the form of $c_f^-$ in \eqref{prop:time-to-spec-bound-cf+}, and gets absorbed into $R_f^-(\lambda_0-\lambda) = o(1)$ in \eqref{eq:spec-rigorous-defn}.

We now investigate the spectral density corresponding to $\gamma_2(h)$, which will also get absorbed into $R_f^\pm$ in \eqref{eq:spec-rigorous-defn}. Since $0 < \veps < d < \frac{1}{2}$, calculations similar to the ones in Appendix \ref{app:subsec-time-to-spec-non-boundary} show that, as $\lambda \to \lambda_0$,
\begin{equation} \label{eq:time-spec-boundary-102}
\begin{split}
 \frac{1}{2\pi}\sum_{h=-\infty}^\infty e^{-ih\lambda} \gamma_{2} (h) &= o(|\lambda - \lambda_0|^{-2(d-\veps)})+ \frac{\Gamma(2(d-\veps))}{2\pi} \left[ L_{1,\gamma} \left( \frac{1}{|\lambda-\lambda_0|}  \right) \cos((d-\veps)\pi)  \right.   \\
 &\quad- \left. \text{sign}(\lambda-\lambda_0) \xi L_{2,\gamma} \left( \frac{1}{|\lambda-\lambda_0|}  \right) \sin((d-\veps)\pi)  \right] |\lambda - \lambda_0|^{-2(d-\veps)} .
\end{split}
\end{equation}
The relations \eqref{eq:time-spec-boundary-101}--\eqref{eq:time-spec-boundary-102} show that $R_f^\pm = o(1)$.
\end{proof}

\subsection{Set of Admissible Cyclical Phases and Special Cases}
\label{subsubsec:admis-phase}

Since $c_f^+,c_f^- \geq 0$ and $c_\gamma>0$, note that \eqref{eq:cfpm-cgamma} yields
\[
-\frac{\pi}{2} \le d\pi - \phi \le \frac{\pi}{2}, -\frac{\pi}{2} \le  \phi - d\pi \le \frac{\pi}{2}
\]
so that, as stated in \eqref{eq:admiss-sets-def} and shown in Remark \ref{r:adm-phi},
\[
\phi \in \cli_d \doteq \left[\left(d - \frac{1}{2} \right)\pi  , \left(\frac{1}{2} - d\right) \pi \right].
\]
When $c_f^+ = c_f^- = c_f$ (symmetric case), we have from \eqref{eq:cgamma-psi},
\[
c_\gamma = 2 \sqrt{2} \Gamma(1-2d) c_f \sin(2 \pi d), \quad \phi = 0,
\]
or, in terms of $c_{1,\gamma},c_{2,\gamma}$, from Remark \ref{rmk:cg-c1g-c2g},
\begin{equation} \label{eq:cf-cf+cf}
c_{1,\gamma} = 4 \Gamma(1-2d) c_f \sin(\pi d), \quad c_{2,\gamma} = 0.
\end{equation}

If $c_f^- = 0$, then Proposition \ref{it:spec-to-time-border} says that
\[
c_\gamma = 2 \Gamma(1-2d) c_f^+, \quad \phi =  \left( \frac{1}{2} - d\right) \pi,
\]
or, through Remark \ref{rmk:cg-c1g-c2g},
\[
c_{1,\gamma} = 2 \Gamma(1-2d) c_f^+ \sin(\pi d), \quad c_{2,\gamma} = - 2 \Gamma(1-2d) c_f^+ \cos(\pi d), \quad \frac{c_{1,\gamma}}{c_{2,\gamma}} = - \tan(\pi d).
\]

Similarly, $c_f^+ = 0$ corresponds to
\[
c_\gamma = 2 \Gamma(1-2d) c_f^-, \quad \phi =  -\left( \frac{1}{2} - d\right) \pi,
\]
or, equivalently,
\[
c_{1,\gamma} = 2 \Gamma(1-2d) c_f^- \sin(\pi d), \quad c_{2,\gamma} = 2 \Gamma(1-2d) c_f^- \cos(\pi d), \quad \frac{c_{1,\gamma}}{c_{2,\gamma}} =  \tan(\pi d).
\]

\section{Auxiliary Lemma} \label{app:sec-auxillary}

We present here results on higher-order behavior of the Fourier series of power-law coefficients that was used in Appendix \ref{app:subsec-time-to-spec-boundary}.

\begin{lemma} \label{lemma:second-order-time-to-spec}
   For $d \in (0,1/2)$, as $\om \to 0^+$,
    \begin{equation}
    \begin{split}
       \sum_{k=1}^\infty \sin(k \om) k^{2d-1} = \om^{-2d} \Gamma(2d) \sin(\pi d) + R_1(\om), \\
       \sum_{k=1}^\infty \cos(k \om) k^{2d-1} = \om^{-2d} \Gamma(2d) \cos(\pi d) + R_2(\om),
    \end{split}
    \end{equation}
    where $R_i(\om) \to c_i \in \RR$, $i=1,2$.
\end{lemma}

\begin{proof}
    Write
    \begin{equation} \label{eq:lemma-aux-1}
    \begin{split}
        \int_0^\infty e^{iu \om} u^{2d-1} du - \sum_{k=1}^\infty e^{ik\om} k^{2d-1} &= \sum_{k=1}^\infty \left( \int_{k-1}^{k} e^{iu \om} u^{2d-1} du -  e^{ik\om} k^{2d-1}  \right) \\
        &=\sum_{k=1}^\infty \left( \int_{k-1}^{k} e^{iu\om} (u^{2d-1} - k^{2d-1})du  + k^{2d-1}  \left( \int_{k-1}^{k} e^{iu\om} du - e^{ik\om}  \right)   \right) .
    \end{split}
    \end{equation}
    Next, note that, by Taylor's expansion,
    \begin{equation}
        \int_{k-1}^{k} (u^{2d-1} - k^{2d-1}) du = k^{2d} \int_{1-1/k}^{1} (z^{2d-1} -1) dz \sim  -k^{2d-2} \frac{2d-1}{2},
    \end{equation}
    which implies,
    from the bounded convergence theorem, as $\om \to 0^+$,
    \begin{equation} \label{eq:lemma-aux-2}
        \sum_{k=1}^\infty \int_{k-1}^{k} e^{iu \om} (u^{2d-1} - k^{2d-1}) du \to - C.
    \end{equation}
    Moreover, 
    \begin{equation}
        \int_{k-1}^{k} e^{iu\om} du - e^{ik\om} = \frac{e^{i\om k } - e^{i\om (k-1)}}{i\om} - e^{i\om k} = e^{i k \om} \frac{1 - e^{- i \om}   -i \om}{i \om}.
    \end{equation}
    Thus,
    \begin{equation} \label{eq:lemma-aux-3}
        \sum_{k=1}^\infty k^{2d-1}  \left( \int_k^{k+1} e^{iu\om} du - e^{ik\om}  \right) = \frac{1 - e^{- i \om}   -i \om}{i \om} \sum_{k=1}^\infty (k^{2d-1} e^{ik \om}) \sim  \om^{-2d+1} \Gamma(2d) e^{i \pi (\frac{1}{2} - d) } = o(1),
    \end{equation}
    where the last asymptotic relation holds from Proposition A.2.1 of \cite{pipiras_taqqu_2017}. 
    By combining \eqref{eq:lemma-aux-1}--\eqref{eq:lemma-aux-3},
    \begin{equation}
    \begin{split}
        \sum_{k=1}^\infty e^{ik\om} k^{2d-1} &=  \int_0^\infty e^{iu \om} u^{2d-1} du  \\
        &\quad-\sum_{k=1}^\infty  \int_{k-1}^{k} e^{iu\om} (u^{2d-1} - k^{2d-1})du  - \sum_{k=1}^\infty k^{2d-1}  \left( \int_{k-1}^{k} e^{iu\om} du - e^{ik\om}  \right)   \\
        &= \om^{-2d} \Gamma(2d) e^{i d \pi}  + R(\om),
    \end{split}
    \end{equation}
    where, as $\om \to 0^+$, $R(\om) \to C - \frac{1}{2d}$ and $C$ is given in \eqref{eq:lemma-aux-2}.
\end{proof}

\end{document}